\documentclass[11pt]{amsart}

\usepackage [PS]{diagrams}

\newcommand\A{\mathfrak{A}}
\newcommand\B{\mathfrak{B}}
\newcommand\C{\mathcal{C}}
\newcommand\D{\mathcal{D}}
\newcommand\I{\mathcal{I}}
\newcommand\J{\mathcal{J}}

\newcommand\nn{\mathbb{N}}
\newcommand\zz{\mathbb{Z}}
\newcommand\cc{\mathbb{C}}

\newcommand{\LL}{\mathbf{L}}
\newcommand{\RR}{\mathbf{R}}
\newcommand{\shv}{\mathfrak{Shv}\,}
\newcommand{\hshv}{\mathfrak{hShv}\,}
\newcommand{\Ho}{\mathrm{Ho}\,}
\newcommand{\Hof}{\mathrm{Ho}_{\mathrm{f}}\,}

\newcommand{\id}{\mathrm{id}}
\newcommand{\Fun}{\mathrm{Fun}\,}
\newcommand{\pr}{\mathrm{pr}}

\newcommand{\Gr}{\widetilde{\hbox {\it Gr}}\,}
\newcommand{\Mod}{\mathrm{Mod}}
\newcommand{\Rng}{\mathrm{Rng}}
\newcommand{\ppp}{\mathbf {pP}}
\newcommand{\weq}{weak equivalence}
\newcommand{\weqs}{weak equivalences}

\newcommand{\op}{\mathrm{op}}

\newcommand{\ie}{{\it i.e.}}
\newcommand{\resp}{{\it resp.}}

\newcommand{\iso}{\cong}

\newcommand{\Top}{\mathbf{Top}}
\newcommand{\DTop}{\hbox{\nobreakdash-}\Top}
\newcommand{\Tw}{\mathbf{Tw}}

\newcommand{\smashover}[1]{\mathop{\wedge}_{#1}}

\newcommand{\invlim}{\lim\limits_{\longleftarrow}\,}
\newcommand{\dirlim}{\lim\limits_{\longrightarrow}\,}

\def\po#1,#2,#3{{#1} \cup_{#2} {#3}}
\def\pb#1,#2,#3{{#1} \times_{#2} {#3}}

\newarrowmiddle{circle}{\circ}{\circ}{\circ}{\circ}
\newarrow{Nonsense}--{circle}->

\newarrow{Equal}=====
\newarrow {Implies}===={=>}
\newarrow {Fib} ----{>>}
\newarrow {Cof} >--->

\newcommand{\ra}{\rTo\relax}
\newcommand{\lra}{\rTo}
\newcommand{\la}{\lTo\relax}
\newcommand{\lla}{\lTo}
\newcommand{\da}{\dTo}

\theoremstyle{plain}
\newtheorem{lemma}{Lemma}[subsection]
\newtheorem{proposition}[lemma]{Proposition}
\newtheorem{theorem}[lemma]{Theorem}
\newtheorem{corollary}[lemma]{Corollary}

\theoremstyle{definition}
\newtheorem{definition}[lemma]{Definition}
\newtheorem{remark}[lemma]{Remark}

\newtheorem{example}[lemma]{Example}

\begin{document}

\title{Twisted diagrams and homotopy sheaves}

\author{Thomas H{\"u}ttemann}

\address{Thomas H{\"u}ttemann\\Queen's University Belfast\\Pure
  Mathematics Research Centre\\Belfast BT7~1NN\\Northern Ireland, UK}

\email{t.huettemann@qub.ac.uk}

\urladdr{http://huettemann.110mb.com/}

\author{Oliver R{\"o}ndigs}

\address{Oliver R{\"o}ndigs, Universit{\"a}t Osnabr{\"u}ck,
Fachbereich Mathe\-matik/In\-for\-ma\-tik,
Albrechtstr.~28a,
D--49069 Osnabr{\"u}ck,
Germany}

\email{oroendig@mathematik.uni-osnabrueck.de}

\urladdr{http://www.mathematik.uni-osnabrueck.de/staff/phpages/rndigso.rdf.html}

\begin{abstract} Twisted diagrams are generalised diagrams: The
  vertices are allowed to live in different categories, and the
  structure maps act through specified ``twisting'' functors between
  these categories. Examples include spectra (in the sense of homotopy
  theory) and quasi-coherent sheaves of modules on an algebraic
  variety. We construct a twisted version of \textsc{Kan} extensions
  and establish various model category structures (with pointwise weak
  equivalences). Using these, we propose a definition of ``homotopy
  sheaves'' and show that a twisted diagram is a homotopy sheaf if and
  only if it gives rise to a ``sheaf in the homotopy category''.
  \hfill(\today)
\end{abstract}

\maketitle

\section{Introduction}

One often encounters constructions which pretend to be diagrams in
some category but cannot quite be described with that formalism.
An important example is the notion of (na\"\i ve) spectra, a sequence of
pointed spaces $ X_0,\ X_1,\ \ldots $ and structure maps
$ \Sigma X_n \ra X_{n+1} $. This almost determines a diagram indexed
over~$\nn$ (regarded as a category), and in fact can be described by a
``twisted diagram'' with ``twists'' given by iterated suspension functors.
Another example (and the origin of the present paper) is the category of
quasi-coherent sheaves on projective spaces as defined by the first
author in~\cite{Hu1}: a ``sheaf'' is a collection of equivariant spaces,
each equipped with an action of a different monoid, together with
structure maps which are equivariant with respect to the ``smaller''
monoid. A detailed description is contained in the examples in this paper.%
---The new formalism also applies, as a special case, to diagram categories
in the usual sense (\ie, functor categories).

To illustrate the general idea, suppose we have two categories~$\C$ and~$\D$
and a functor $ F \colon \C \ra \D $ which has a right adjoint~$U$.
A twisted diagram (with respect to this data) is a morphism (in~$\D$)
\[F(Y) \lra^{f^\sharp} Z\]
where~$Y$ is an object of~$\C$ and $Z$~is an object of~$\D$. This gadget
should be thought of as a generalised diagram of the form $ Y \rNonsense^f Z $.
Since~$Y$ and~$Z$ live in different categories, the ``structure map''~$f$
has to act by a ``twist'' given by~$F$.

The twisted diagram $Y \rNonsense Z$ is a {\it strict sheaf\/} if its
structure map $F(Y) \rTo Z$ is an isomorphism. The terminology is
motivated by the description of quasi-coherent sheaves on a
quasi-compact scheme by twisted diagrams of modules: The category of
quasi-coherent sheaves is equivalent to the full subcategory of those
twisted diagrams which are strict sheaves;
cf.~Example~\ref{example:quasi_coh_toric}.---Provided the
categories~$\C$ and~$\D$ are equipped with compatible model structure,
we can define the twisted diagram to be a {\it homotopy sheaf\/} if
the structure map $F(Y) \rTo Z$ is a weak equivalence (the technical
definition given in~\ref{def:homotopy_sheaves} involves a cofibrant
replacement of the source). The rough idea is that, by passing to the
homotopy class of the structure map, we obtain a twisted diagram
involving the homotopy categories of~$\C$ and~$\D$ only, and that $Y
\rNonsense Z$ is a homotopy sheaf if and only if it is mapped to a
strict sheaf in the homotopy categories. However, there are technical
issues which make this process of ``passing to the homotopy
categories'' slightly more complicated than one would wish; these
issues are addressed in~\S4.

Homotopy sheaves of the kind described here appeared in the context of
the algebraic $K$-theory of spaces \cite{HKVWW, Hu1, Hu2,
  Hu3}. However, none of these papers addressed the question of the
exact relationship between sheaves and homotopy sheaves, and the
present paper is intended to provide clarification of this point (see
the remarks at the very end of this paper). In a forthcoming paper, it
will be shown that homotopy sheaves can be used to describe the
derived category, in the usual sense, of toric varieties, and that
homotopy sheaves can be characterised as colocal objects in a twisted
diagram category~\cite{Hu4}.

\goodbreak

To make this rather abstract paper accessible to a broad audience, we
include a quite detailed discussion of elementary topics, in
particular in the foundational material. This includes completeness of
the category of twisted diagrams and the construction of {\author Kan}
extensions.  Apart from some elementary category theory~\cite{ML} and
and basic model category theory \cite{DS, Ho} no prerequisites are
required. While the model structures are well-known among experts in
homotopy theory, it seems that there has been no accessible published
account so far. We hope to fill the gap with the present paper.

{\bf Organisation of the paper.} The paper is divided into three
parts. \S\ref{sec:foundations} is devoted to the definition of twisted
diagrams and the development of basic machinery. The fundamental
notion is that of an {\it adjunction bundle}, consisting of a
collection of categories and adjoint functor pairs.  It encodes the
shape of the diagrams and carries all the necessary information about
twists. We discuss the behaviour of twisted diagrams with respect to
morphisms of adjunction bundles and prove a convenient criterion for
completeness. In \S\ref{subsec:kan} we construct a twisted version of
{\author Kan\/} extensions. Section~\ref{subsec:construction} includes
a different description of twisted diagrams and shows how to construct
important examples of adjunction bundles.

In \S\ref{sec:model-structures} we prove the existence of several
{\author Quillen\/} closed model category structures on categories of
twisted diagrams. This part is based on model category structures for
diagram categories as in~\cite{Ho}.  In more detail, we consider
``pointwise'' weak equivalences.  Depending on properties of the
adjunction bundle (the index category is required to be a ``direct''
or ``inverse'' category), we establish {\author Reedy\/}-type model
structures using (generalised) latching or matching spaces.  If the
adjunction bundle consists of cofibrantly generated model categories,
we construct (for arbitrary index categories) a cofibrantly generated
model structure.

Finally, in~\S\ref{sec:sheaves} we propose definitions of sheaves and
homotopy sheaves. Starting from an adjunction bundle of model
categories we construct an associated bundle of homotopy categories. A
twisted diagram over the original adjunction bundle gives rise to a
twisted diagram over the homotopy bundle, and the former is a homotopy
sheaf if and only if the latter is a sheaf.

A special case of the results on model structures has been used by the
first author to study the algebraic $K$-theory of projective toric
varieties \cite{Hu1, Hu3}. Homotopy sheaves are important to control
finiteness of homotopy colimits of infinite $CW$
complexes~\cite{Hu2}. Twisted diagrams and their model structures also
appear implicitly in~\cite{HKVWW}.

\section {Foundations}
\label{sec:foundations}

\subsection {Adjunction Bundles}

Let~$\I$ be a small category. It will serve as the index category for our diagrams.

\begin{definition}
  \label{DefAdjBundle}
  An {\it adjunction bundle~$\B = (\C,\ F,\ U)$ over~$\I$\/}, or
  $\I$-bundle, consists of the following data:
  \begin{itemize}
  \item for each object $i \in \I$ a category $\C_i$,
  \item for each morphism $\sigma \colon i \ra j$ in~$\I$ a pair of
    adjoint functors
    \[F_\sigma \colon \C_i \ra C_j \qquad \hbox{and} \qquad U_\sigma
    \colon \C_j \ra \C_i\]
    (with $F_\sigma$ being the left adjoint),
  \end{itemize}
  such that~$U$ determines a functor $\I^\op \ra Cat$, \ie, $U_{\id_i}
  = \id_{\C_i}$, and for each pair of composable arrows $i \lra^\sigma
  j \lra^\tau k$, the equality $U_{\tau \circ \sigma} = U_\sigma \circ
  U_\tau$ holds. In addition, we require $F_{\id_i} = \id_{\C_i}$. The
  properties of adjunctions guarantee that there is a canonical
  isomorphism $F_{\tau \circ \sigma} \iso F_\tau \circ F_\sigma$
  (which will be referred to as {\it uniqueness isomorphism\/}), since
  both functors are left adjoint to $U_{\tau \circ \sigma} = U_\sigma
  \circ U_\tau$ (\cite[IV.1, Corollary~1, p.~83]{ML}).
\end{definition}

\begin{example}
  \label{XampleTrivial}
  Any category~$\C$ gives rise to a {\it trivial $\I$-bundle\/} with
  $\C_i = \C$ for all~$i$, and all adjunctions being the identity
  adjunction.
\end{example}

\begin{example}[The non-linear projective line]
  \label{TheProjectiveLine}
  Let $M \DTop$ be the category of pointed topological spaces with a
  basepoint-preserving action of the monoid~$M$.  A monoid
  homomorphism $f\colon M \ra M^\prime$ determines an adjunction
  \[\cdot \smashover {M} M^\prime \colon
  M \DTop \pile{\ra \\ \la} M^\prime \DTop \colon R_f\]
  with~$R_f$ being restriction along~$f$, and~$ \cdot \smashover {M}
  M^\prime$ being its left adjoint (inducing up). The integers~$\zz$
  form a monoid under addition, and we have sub-monoids~$\nn_+$
  (non-negative integers) and~$\nn_-$ (non-positive integers). Hence
  we can form the adjunction bundle $\mathfrak {P}^1$ over $\I = (+
  \lra^{\alpha} 0 \lla^{\beta} -)$, consisting of the categories
  $\nn_+ \DTop$, $\zz \DTop $ and~$\nn_- \DTop$, and the adjoint pairs
  ``inducing up'' and ``restriction'' along the inclusions $\nn_+
  \subseteq \zz$ and $\nn_- \subseteq \zz$.
\end{example}

\subsection{Twisted Diagrams}

\begin{definition}[Twisted diagrams]
  \label{DefTwistedDiag}
  Let~$\B$ be an adjunction bundle over~$\I$.  A {\it twisted
    diagram~$Y$ with coefficients in $\B$\/} consists of the following
  data:
  \begin{itemize}
  \item for each object $i \in \I$ an object $Y_i \in \C_i$,
  \item for each morphism $\sigma \colon i \ra j $ in~$\I$ a map
    $y_\sigma^\flat \colon Y_i \ra U_\sigma (Y_j) $ in~$ \C_i$
  \end{itemize}
  such that $Y$ behaves like a functor, \ie, $y^\flat_{\id_i}
  = \id_{Y_i} $ and $ y_{\tau \circ \sigma}^\flat = U_\sigma
  (y_\tau^\flat) \circ y_\sigma^\flat $ for each pair $ i \lra^\sigma
  j \lra^\tau k $ of composable arrows in $\I$.  (A reformulation
  using the left adjoints will be given below.)
        
  A {\it map $f \colon Y \ra Z$ of twisted diagrams\/} is a
  collection of maps
  \[f_i \colon Y_i \ra Z_i\]
  in $\C_i$, one for each object $i \in \I$, such that for each
  morphism $\sigma \colon i \ra j$ in~$\I$ the equality $U_\sigma
  (f_j) \circ y_\sigma^\flat = z_\sigma^\flat \circ f_i$ holds. (A
  reformulation using the left adjoints will be given below.)  The
  category of twisted diagrams and their maps is denoted $\Tw
  (\I,\B)$.
\end{definition}

For each of the structure maps $y_\sigma^\flat \colon Y_i \ra U_\sigma
(Y_j)$ there is a corresponding adjoint map $y_\sigma^{\sharp} \colon
F_\sigma (Y_i) \ra Y_j$.  The idea is to think of the (meaningless)
symbol $y_\sigma \colon Y_i \rNonsense Y_j$ as a kind of ``structure
map'' having two incarnations as a $\flat$-type map (a morphism
in~$\C_i$) and a $\sharp$-type map (a morphism in~$\C_j$).

The definition of twisted diagrams does not make explicit use of the
left adjoints provided by the adjunction bundle. However, the
properties of adjunctions will play a crucial r\^ole for the
discussion of limits and colimits in~$\Tw (\I, \B)$.

\begin{example}[Spectra]
  \label{Spectra}
  Let~$\nn$ denote the ordered set of natural numbers, considered as a
  category. For each $n \in \nn$, define~$\C_n$ to be the
  category~$\mathcal{S}$ of pointed simplicial sets. If $n\leq m$, we
  have an adjunction $\Sigma^{m-n} \colon \mathcal{S} \pile{\ra \\
    \la} \mathcal{S} \colon \Omega^{m-n}$ of iterated loop space and
  suspension functors. It is clear that this defines an adjunction
  bundle~$Sp$ over~$\nn$. A twisted diagram~$X$ with coefficients
  in~$Sp$, graphically represented by the ``diagram''
  \[X_0 \rNonsense X_1 \rNonsense X_2 \rNonsense \ldots \ ,\]
  is nothing but a spectrum in the sense of \textsc{Bousfield} and
  \textsc{Friedlander}, cf.~\cite{BF}.
\end{example}

\begin{remark}
  \begin{enumerate}
  \item If~$\B$ is a trivial $\I$-bundle
    (Example~\ref{XampleTrivial}), we recover the functor category:
    $\Tw (\I, \B) = \Fun (\I,\ \C)$.
  \item If~$\I$ is a discrete category (\ie, contains no non-identity
    morphisms), an adjunction bundle over~$\I$ is nothing but a
    collection of categories~$\{\C_i\}_{i\in \I}$, and the category of
    twisted diagrams is the product category $\prod_{i\in \I}\C_i$.
  \item Suppose $\B_\nu = (\C^\nu,\ F^\nu,\ U^\nu) $ is a family of
    adjunction bundles indexed by $\I_\nu$. Then we can form the
    following adjunction bundle $\prod_\nu \B_\nu =: \B = (\C,\ F,\
    U)$ indexed by the disjoint union $\I:= \amalg_\nu \I_\nu $: for
    each $i \in \I$ there is exactly one~$ \nu $ with $ i \in \I_\nu$,
    and we define $\C_i = \C_i^\nu$ (and similarly for the~$F$
    and~$U$).  It is easy to see that $\Tw (\I,\B) = \prod_\nu \Tw
    (\I_\nu,\B_\nu)$ in this case.
  \end{enumerate}
\end{remark}

\goodbreak

Given twisted diagrams $Y, Z \in \Tw (\I, \B)$ and maps $ f_i \colon
Y_i \ra Z_i $ in~$\C_i$, we can form two squares for each morphism
$\sigma \colon i \ra j $ in~$ \I$
\[\begin{diagram}
        Y_i & \lra^{f_i} & Z_i \\
        \da<{y_\sigma^\flat} && \da>{z_\sigma^\flat} \\
        U_\sigma (Y_j) & \lra[l>=4em]_{U_\sigma (f_j)} & U_\sigma (Z_j) \\
\end{diagram}
\qquad {\rm and} \qquad
\begin{diagram}
        F_\sigma (Y_i) & \lra[l>=4em]^{F_\sigma (f_i)} & F_\sigma (Z_i) \\
        \da<{y_\sigma^\sharp} && \da>{z_\sigma^\sharp} \\
        Y_j & \lra_{f_j} & Z_j \\
\end{diagram}\]
and the definition of adjunctions imply that the left square commutes
if and only if the right square commutes. Thus the family~$(f_i)_{i
  \in \I}$ determines a map of twisted diagrams if and only if
$z_\sigma^\sharp \circ F_\sigma (f_i) = f_j \circ y_\sigma^\sharp$.

For later use, we record the following fundamental fact:

\begin{lemma}
  \label{DiagramCond}
  Suppose we have a map $y_\sigma^\flat \colon Y_i \ra U_\sigma (Y_j)$
  in~$\C_i$ for each morphism $\sigma \colon i \ra j$ in~$\I$
  satisfying $y_{\id}^\flat = \id$, and denote by~$ y_\sigma^\sharp$
  the adjoint map $F_\sigma (Y_i) \ra Y_j $.  Let $ \tau \colon j \ra
  k $ be another morphism in~$\I$. Then if one of the squares
  \[\begin{diagram}
    F_\tau \circ F_\sigma (Y_i) & \lra[l>=2.5em]^{\iso} & F_{\tau \circ \sigma} (Y_i) \\
    \da<{F_\tau (y_\sigma^\sharp)} && \da>{y_{\tau \circ \sigma}^\sharp} \\
    F_\tau (Y_j) & \lra_{y_\tau^\sharp} & Y_k \\
  \end{diagram}
  \qquad {\rm \it and} \qquad
  \begin{diagram}
    Y_i & \rEqual & Y_i \\
    \da<{y_\sigma^\flat} && \da>{y_{\tau \circ \sigma}^\flat} \\
    U_\sigma (Y_j) & \lra[l>=4em]_{U_\sigma (y_\tau^\flat)}
    & U_{\tau \circ \sigma} (Y_k) \\
  \end{diagram}\]
  commutes so does the other (the upper horizontal map in the left
  square is the uniqueness isomorphism). In other words, if for all
  composable morphisms~$\sigma$ and~$\tau$ one of the squares
  commutes, the objects~$ Y_i $ together with the
  maps~$y_\sigma^\flat$ form a twisted diagram.
\end{lemma}

\begin{proof}
  This follows from naturality of and composition rules for units and
  counits of adjunctions, cf.~\cite[\S\S IV.1+8]{ML}. We omit the
  details.
\end{proof}

\subsection{Limits, Colimits, Direct and Inverse Image}

The next proposition says that {\it $\Tw (\I,\B)$ is as complete and
  cocomplete as all the~$\C_i$\/}, and that limits {\resp, colimits}
can be computed ``pointwise'' in the categories~$\C_i$.  For $i \in
\I$, let $Ev_i \colon \Tw (\I,\B) \ra \C_i$ denote the $i$th
evaluation functor which maps a twisted diagram $Y$ to its $i$th
term~$Y_i$.

\begin{proposition}[Limits and colimits of diagrams of twisted diagrams]
  \label{Completeness}
  Let~$ G \colon \D \ra \Tw (\I,\B) $ be a functor, and suppose that
  for all~$i$ the limit of~$ Ev_i \circ G $ exists. Then~$ \invlim G $
  exists and the canonical map
  \[Ev_i (\invlim G) \ra \invlim (Ev_i \circ G)\]
  is an isomorphism. A similar assertion holds for colimits.
\end{proposition}

\begin{proof}
  The proof relies on the compatibility of left (\resp, right) adjoint
  functors with colimits (\resp, limits): if~$F$ is a left adjoint,
  and~$D$ is a functor, then there is a unique natural isomorphism $
  \dirlim (F \circ D) \ra F (\dirlim D) $, and similarly for right
  adjoints and limits (\cite[\S V.5, Theorem~1, p.~114]{ML}).

  To prove the lemma, we treat the case of colimits only.  (For limits
  one has to use similar techniques. Since~$U$ is supposed to be
  functorial, this is slightly easier.) Let $G_i := Ev_i \circ G$, and
  define $C_i := \dirlim G_i$. We claim that the objects~$C_i$
  assemble to a twisted diagram~$C$, and it is almost obvious that~$C$
  is ``the'' colimit of~$G$.

  Let $ \sigma \colon i \ra j $ denote a morphism in~$\I$.  The
  $\sharp$-type structure maps of the twisted diagrams~$ G (d) $ (for
  objects~$d \in \D$) assemble to a natural transformation
  \[G_\sigma^\sharp \colon F_\sigma \circ G_i \ra G_j\]
  of functors $\D \ra \C_j$. Hence we can define the $\sharp$-type
  structure map~$ c_\sigma^\sharp $ as the composite
  \[F_\sigma (C_i) = F_\sigma (\dirlim G_i) \iso \dirlim (F_\sigma
  \circ G_i) \lra^f \dirlim G_j = C_j\]
  with~$f$ induced by~$ G_\sigma^\sharp$.

  By Lemma~\ref{DiagramCond} we are left to show that the following
  square commutes for all composable morphisms~$\sigma $ and~$ \tau$
  in~$\I$: \diagram[eqno=(*)]
  F_\tau \circ F_\sigma (C_i) & \lra[l>=2.5em]^{\iso} & F_{\tau \circ \sigma} (C_i) \\
  \da<{F_\tau (c_\sigma^\sharp)} && \da>{c_{\tau \circ \sigma}^\sharp} \\
  F_\tau (C_j) & \lra_{C_\tau^\sharp} & C_k \\
        \enddiagram
        We replace the symbols~$ C_\ell $ and the structure maps by their definition and obtain
        the following bigger diagram:
        \diagram[eqno=(**)]
                F_\tau \circ F_\sigma (\dirlim G_i) &
                        \rEqual & F_\tau \circ F_\sigma (\dirlim G_i) &
                        \lra[l>=2.5em]^\iso & F_{\tau \circ \sigma} (\dirlim G_i) \\
                \da<\iso & 1 & \da<\iso & 2 & \da>\iso \\
                F_\tau \bigl( \dirlim (F_\sigma \circ G_i)\bigr) & \lra[l>=2.5em]^\iso &
                        \dirlim (F_\tau \circ F_\sigma \circ G_i) & \lra[l>=2.5em]^\iso &
                        \dirlim (F_{\tau \circ \sigma} \circ G_i) \\
                \da & 3 & \da & 4 & \da \\
                F_\tau \bigl( \dirlim (G_j)\bigr) & \lra[l>=2.5em]_\iso &
                        \dirlim (F_\tau \circ G_j) & \ra & \dirlim (G_k) \\
        \enddiagram
        All the small squares commute: for~1 this is true by uniqueness of the isomorphisms
        for commuting left adjoints with colimits. The horizontal maps of~2 are induced by the
        uniqueness isomorphism, the vertical maps are induced by the isomorphism for
        commuting left adjoints with colimits. By uniqueness, 2~commutes.
        Both horizontal maps of~3 are induced by the isomorphism for commuting colimits 
        with~$ F_\tau $,
        and both vertical maps are induced by the natural transformation
        $ G_\sigma^\sharp \colon F_\sigma       \circ G_i \ra G_j $. Hence~3 commutes.
        Finally, square~4 commutes by Lemma~\ref{DiagramCond}, applied
        componentwise to
        the diagrams~$G_\ell$, and by functoriality of~$ \dirlim $.

        Hence the diagram~$(**)$ commutes.
        But the square~$(*)$ is contained in there as the outer square,
        thus is commutative as claimed.
\end{proof}

If $\I$ is a small category and $\C$ is an arbitrary category, the category of
diagrams $\Fun(\I,\C)$ is the value of an internal hom functor on the category of 
categories. Hence it is functorial in both variables (provided the entries in
the first variable are small). The case of twisted diagrams is more
complicated since it involves ``maps'' of adjunction bundles as well as of the
actual diagrams.

\begin{definition}[Inverse image of bundles]
  Given a functor $\Phi \colon \I \ra \J $ and a $\J$-bundle $\B =
  (\D,\ G,\ V)$, we define the {\it inverse image of~$\B$
    under~$\Phi$\/}, denoted~$\Phi^* \B$, to be the $\I$-bundle $(\C,\
  F,\ U)$ given by $\C_i := \D_{\Phi (i)}$, $U_i := V_{\Phi (i)}$ and
  $F_i := G_{\Phi (i)}$.

  If $\Phi \colon \I \ra \J$ is the inclusion of a subcategory, we
  write $\B|_{\I}$ instead of $\Phi^* \B $ and call the resulting
  $\I$-bundle the {\it restriction\/} of~$\B$ to~$\I$.
\end{definition}

Forming inverse images is functorial, \ie, $\id_\C^* \B = \B$ and
$(\Phi \circ \Theta)^* \B = \Theta^* \Phi^* \B$.  The inverse image of
a trivial bundle is a trivial bundle.

\begin{definition}[Morphisms of bundles]
  Suppose $\A = (\C,\ F,\ U)$ and $\B = (\D,\ G,\ V)$ are
  $\I$-bundles. An {\it $\I$-morphism\/} $\Psi \colon \A \ra \B$
  consists of two families of functors $ \rho_i \colon \C_i \ra \D_i$
  and $\lambda_i \colon \D_i \ra \C_i$ where~$i$ ranges over the
  objects of~$\I$ such that $\lambda_i$~is left adjoint to~$\rho_i$,
  and such that for each morphism $\sigma \colon i \ra j$ in~$\I$ we
  have $V_\sigma \circ \rho_j = \rho_i \circ U_\sigma$.

  Given an $\I$-bundle~$\A$ and a $\J$-bundle~$\B$, a {\it morphism of
    bundles \/}
  \[\Xi \colon \A \ra \B\]
  is a pair $\Xi = (\Phi,\ \Psi)$ where $\Phi \colon \I \ra \J$ is a
  functor and $\Psi \colon \A \ra \Phi^* \B$ is an $\I$-morphism of
  $\I$-bundles.
\end{definition}

\begin{definition}[Inverse image of twisted diagrams]
  \label{DefInvDiag}
  Suppose we have a functor $\Phi \colon \I \ra \J$, a
  $\J$-bundle~$\B$, and a twisted diagram $Y \in \Tw (\J, \B)$. We
  define the {\it inverse image of~$Y$ under~$\Phi$\/},
  denoted~$\Phi^* Y$, as the twisted diagram over~$\I$ with
  coefficients in~$\Phi^* \B$ given by $(\Phi^* Y)_i := Y_{\Phi (i)}$
  and $(\Phi^* y)_\sigma^\flat := y_{\Phi (\sigma)}^\flat$ for all
  objects $i \in \I$ and all morphisms $ \sigma \in \I$.  We obtain a
  functor $\Phi^* \colon \Tw (\J,\B) \ra \Tw (\I,\Phi^*\B)$.

  Now suppose we have $\I$-bundles $\A = (\C,\, F,\, U)$ and $\B =
  (\D,\, G,\, V)$, and an $\I$-morphism $\Psi = (\rho,\, \lambda)
  \colon \A \ra \B$.  The functor {\it inverse image under~$\Psi$\/},
  denoted $\Psi^* \colon \Tw (\I,\B) \ra \Tw (\I,\A)$, assigns to a
  twisted diagram $ Y \in \Tw (\I,\B) $ the object $\Psi^* Y \in \Tw
  (\I,\A)$ given by $(\Psi^* Y)_i := \lambda_i (Y_i)$ with
  $\sharp$-type structure maps $(\Psi^* y)_\sigma^\sharp$ given by
  the composition
  \[F_\sigma ((\Psi^* Y)_i) = F_\sigma (\lambda_i (Y_i)) \iso
  \lambda_j (G_\sigma (Y_i)) \lra^{\lambda_i (y_\sigma^\sharp)}
  \lambda_j (Y_j) = (\Psi^* Y)_j\]
  for all objects~$i \in \I$ and morphisms $\sigma \colon i \ra j$.
  (We will prove in the next lemma that $\Psi^*$ is well-defined, \ie,
  that $\Psi^* Y$ is a twisted diagram.)

  More generally, a morphism $\Xi = (\Phi,\ \Psi) \colon \A \ra \B$ of
  bundles induces an inverse image functor $\Xi^* = \Psi^* \circ
  \Phi^* \colon \Tw (\J,\B) \ra \Tw (\I,\A)$.
\end{definition}

If $\Phi \colon \I \ra \J$ is the inclusion of a subcategory, we write
$Y|_{\I}$ instead of $\Phi^* Y$ and call the resulting twisted diagram
with coefficients in $\B|_{\I}$ the {\it restriction of~$Y$
  to~$\I$\/}. This defines the restriction functor
\[\Tw (\J,\B) \ra \Tw (\I,\B|_{\I}) \ .\]
As a special case of restriction (if $\I = \{i\}$ is the trivial
subcategory consisting of $i$), we obtain the evaluation functors
$Ev_i$ as defined above.

\begin{lemma}
  \label{InverseImageIsAFunctor}
  Given $\I$-bundles $\A = (\C,\ F,\ U)$ and $\B = (\D,\ G,\ V)$, an
  $\I$-morphism $\Psi = (\lambda,\rho) \colon \A \ra \B$, and a
  twisted diagram $Y \in \Tw (\I, \B)$, the object $\Psi^* Y$
  defined in Definition~\ref{DefInvDiag} is a twisted diagram with
  coefficients in~$\A$.
\end{lemma}

\begin{proof}
        Let $\sigma \colon i \ra j$ and $\tau \colon j \ra k$ be morphisms in $\I$ and consider
        the diagram 
        \diagram
                F_\tau \circ F_\sigma \circ \lambda_i(Y_i) && \lra^\iso && F_{\tau \circ \sigma} \circ \lambda_i(Y_i)\\
                \da<\iso                                   &&           &&                              \da>\iso     \\
                F_\tau \circ \lambda_j \circ G_\sigma (Y_i)& \lra^\iso & \lambda_k \circ G_\tau \circ G_\sigma (Y_i)  &
                                \lra^\iso       & \lambda_k \circ G_{\tau \circ \sigma} (Y_i)   \\
                \da<{F_\tau\circ\lambda_j(y^\sharp_\sigma)}&           & \da<{\lambda_k\circ G_\tau (y^\sharp_\sigma)}&
                                                & \da>{\lambda_k (y^\sharp_{\tau \circ \sigma})}\\
                F_\tau \circ \lambda_j (Y_j)               & \lra_\iso & \lambda_k \circ G_\tau (Y_j)                 & 
                                \lra_{\lambda_k (y^\sharp_\tau)} & \lambda_k(Y_k)
        \enddiagram
        in which all arrows labelled with ``$\iso$'' denote uniqueness isomorphisms.
        Recall that the compositions of functors appearing in the upper rectangle
        are left adjoints to the functor $U_\sigma \circ U_\tau \circ \rho_k$. Thus
        the upper rectangle commutes by uniqueness. The lower left square commutes
        by naturality. The lower right square commutes since~$Y$ is a twisted
        diagram (Lemma~\ref{DiagramCond}) and $\lambda_k$ is a functor. Hence the whole diagram
        commutes and $\Psi^* Y$~is a twisted diagram by another
        application of Lemma~\ref{DiagramCond}.
\end{proof}

\begin{definition}[Direct image of twisted diagrams]
  Suppose we have a bundle morphism $\Xi = (\Phi,\Psi) \colon \A \ra
  \B$, where $\A = (\C,\ F,\ U)$ is an $\I$-bundle, $\B = (\D,\ G,\
  V)$ is a $\J$-bundle, $\Phi$~is a functor $ \I \ra \J$, and $\Psi =
  \{(\lambda_i,\rho_i)\}_{i \in \I}$ is an $\I$-morphism $\A \ra
  \Phi^*\B$.  Let~$Y$ be a twisted diagram with coefficients
  in~$\A$. It is straightforward to check that the definition
  $\Psi_*(Y)_i := \rho_i(Y_i)$ yields a twisted diagram with
  coefficients in~$\Phi^*\B$ having the structure maps
  \[\Psi_*(y)^\flat_\alpha \colon \rho_i(Y_i)
  \lra^{\rho_i(y^\flat_\alpha)} \rho_i\circ U_\alpha (Y_j) =
  V_{\Phi(\alpha)} \circ \rho_j(Y_j)\]
  for $\alpha \colon i \ra j$. In this way we obtain a functor
  \[\Psi_* \colon \Tw (\I,\A) \ra \Tw (\I,\Phi^*\B) \ .\]
  Suppose the right adjoint~$R\Phi$ of~$\Phi^*$ exists. The
  composition
  \[\Xi_* := R\Phi \circ \Psi_* \colon \Tw (\I,\A) \ra \Tw (\J,\B)\]
  is called the {\it direct image\/} functor.
\end{definition}

We will see below that if the bundle~$\B$ consists of complete
categories, the functor~$R \Phi$ exists and can be constructed by
twisted {\author Kan} extension. Using this, we can prove:

\begin{corollary}
  \label{DirectInverse}
  Let $\Xi = (\Phi,\Psi) \colon \A \ra \B$ be a bundle morphism, with
  $\B$ consisting of complete categories. Then the functor $\Xi^*$
  (inverse image under $\Xi$) has a right adjoint $\Xi_*$ (direct
  image under $\Xi$).
\end{corollary}

\begin{proof}
  Since~$R\Phi$ is right adjoint to~$\Phi^*$ by assumption, it remains
  to show that~$\Psi_*$ is right adjoint to~$\Psi^*$. However, this is
  true because~$\Psi_*$ is pointwise right adjoint to~$\Psi^*$, and
  it can be checked that adjoining pointwise respects maps of twisted
  diagrams. We omit the details.
\end{proof}

\subsection {Twisted Kan Extensions}
\label{subsec:kan}

Assume that~$\B$ is a trivial bundle over~$\J$, consisting of the
category~$\C$ (and identity functors), and $\Phi \colon \I \ra \J$ is
a functor.  In this case, the inverse image of~$\B$ under~$\Phi$ is
the trivial bundle over~$\I$ (consisting of~$\C$ and identity
functors), and~$\Phi^*$ is the functor $\Fun(\J,\C) \ra \Fun(\I,\C)$
mapping~$Y$ to~$Y\circ \Phi$. If~$\C$ is complete, the
functor~$\Phi^*$ has a right adjoint given by right \textsc{Kan\/}
extension along $\Phi$ (\cite[\S X.3, Corollary~2]{ML}).

It is possible to construct \textsc{Kan\/} extensions in our
framework.  We consider only left \textsc{Kan\/} extensions, the other
case being similar (and easier).

\medbreak

Let $\Phi \colon \I \ra \J$ be a functor, $\B=(\C,\ F,\ U)$ a
$\J$-bundle, and $Y$ a twisted diagram over $\I$ with coefficients in
$\Phi^*\B = (\D,\ G,\ U)$.  First, we have to define a twisted diagram
$L(Y)$ over $\J$ with coefficients in $\B$. (Later, we will convince
ourselves that the assignment $Y \rMapsto L(Y)$ is a functor which is
left adjoint to $\Phi^*$.)  Let $j \in \J$ be given, and let $\Phi
\downarrow j$ denote the category of objects $\Phi$-over $j$. Its
objects are maps of the form $\sigma \colon \Phi(i) \ra j \in \J$ (for
$i$ an object of $\I$). The morphisms from $\sigma \colon \Phi(i) \ra
j$ to $\tau \colon \Phi(i') \ra j$ are morphisms $\alpha \colon i \ra
i' \in \I$ satisfying $\tau \circ \Phi(\alpha) = \sigma$.  Consider
the assignment
\[D_j^Y \colon \Phi \downarrow j \ra \C_j, \qquad (\Phi(i) \lra^\sigma
j) \mapsto F_\sigma(Y_i)\]
This is well-defined because $Y_i$ is an object of $\D_i =
\C_{\Phi(i)}$ by definition of $\Phi^*\B$, so $F_\sigma(Y_i)$ is an
object of $\C_j$.

The assignment~$D^Y_j$ is in fact a functor, as one can deduce as follows.
Let $\pr_\I$ denote the obvious projection functor $\Phi \downarrow j \ra \I$
mapping the object $ \Phi (i) \ra j $ to~$i$, and define $\pr_\J := \Phi \circ \pr_\I$.
Using the equality $\pr_\J^*\B = \pr_\I^*(\Phi^*\B)$,
we get a functor $\pr_\I^* \colon \Tw (\I,\Phi^*\B) \ra 
\Tw (\Phi \downarrow j, \pr_\J^*\B)$.
Let $\{j\}$ denote the subcategory of $\J$ given by
the object $j$ (and no non-identity morphism) and consider the category
$\C_j$ as a (trivial) bundle over $\{j\}$. Then we have a morphism of bundles
$\Xi \colon \C_j \ra \pr_\J^*\B$ consisting of the functor 
$\Phi \downarrow j \ra \{j\}$ and the $(\Phi \downarrow j)$-morphism $\Psi$ from
$\pr_\J^*\B$ to the trivial bundle with $\sigma$-component the adjunction 
$F_\sigma \colon \C_{\Phi(i)} \pile{\ra \\ \la} \C_j \colon U_\sigma$
(for $\sigma \colon \Phi(i) \ra j)$. The inverse image under $\Psi$ is a functor
\[\Psi^* \colon \Tw (\Phi\downarrow j, \pr_\J^*\B) \ra \Fun (\Phi
\downarrow j, \C_j) \ .\]
Tracing the definitions shows $D^Y_j =\Psi^* \pr_\I^* (Y)$.

\medskip

Now assume that the bundle $\B$ consists of cocomplete categories. 
Define $L(Y)_j$ as the colimit of $D^Y_j$. To prove that the $L(Y)_j$ assemble
to a twisted diagram, we construct for each $\alpha \colon j \ra k$
a structure map $l_\alpha^\sharp \colon F_\alpha(L(Y)_j) \ra L(Y)_k$ and
apply Lemma~\ref{DiagramCond}.

Since $F_\alpha$ is a left adjoint, we have a unique isomorphism 
$$ u_\alpha\colon F_\alpha(\dirlim D^Y_j) \iso \dirlim (F_\alpha \circ D^Y_j) \ . $$
Let $\Phi(i) \lra^\sigma j$ be an object of $\Phi \downarrow j$. Then
$\alpha \circ \sigma$ is an object of $\Phi \downarrow k$, and there is a canonical 
map $F_{\alpha \circ \sigma}(Y_i) \ra \dirlim D^Y_k = L(Y)_k$
(since~$F_{\alpha \circ \sigma} (Y_i)$ appears in the diagram~$D^Y_k$).
The composition with a uniqueness isomorphism yields a map 
        $$ t_{\sigma} \colon F_\alpha\circ F_\sigma(Y_i) \ra L(Y)_k \ .$$
The $t_\sigma$ assemble to a natural transformation from $F_\alpha \circ D^Y_j$
to the constant diagram with value $L(Y)_k$ (a proof involves the uniqueness of the
uniqueness isomorphisms and the naturality of the canonical maps mentioned
above; we omit the details).
By taking colimits, this determines a map
$$ v_\alpha \colon \dirlim (F_\alpha \circ D^Y_j) \ra L(Y)_k \ , $$
and we set $l^\sharp_\alpha := v_\alpha \circ u_\alpha$.

Now we have to check that, for $j \lra^\alpha k \lra^\beta l \in \J$, the square
\diagram[eqno=(*)]
        F_\beta \circ F_\alpha (L(Y)_j) & \lra[l>=2.5em]^{\iso} & F_{\beta \circ \alpha}(L(Y)_j)\\
        \da<{F_\beta (l_\alpha^\sharp)} && \da>{l_{\beta \circ \alpha}^\sharp} \\
        F_\beta (L(Y)_k)                & \lra_{l_\beta^\sharp} & L(Y)_l \\
\enddiagram
commutes. First of all, the diagram
\diagram[l>=2em]
        F_\beta \circ F_\alpha (L(Y)_j) & \lra^\iso     & F_{\beta \circ \alpha}(L(Y)_j)        \\
        \da<\iso                        &               & \da>\iso                              \\
        \dirlim (F_\beta \circ F_\alpha \circ D^Y_j)&\lra_\iso&\dirlim (F_{\beta \circ \alpha} \circ D^Y_j)
\enddiagram
consisting of uniqueness isomorphisms commutes because of their uniqueness. 
By the universal property of the colimit and the definition of the structure maps, 
we are left to show that for every $\sigma \colon \Phi(i) \ra j$ the diagram
\diagram[l>=3em]
        F_\beta \circ F_\alpha \circ F_\sigma(Y_i)& \lra^\iso& F_{\beta \circ \alpha} \circ F_\sigma(Y_i)\\
        \da<\iso                                &               & \da>\iso                              \\
        F_\beta \circ F_{\alpha \circ \sigma}(Y_i)& \lra_\iso& F_{\beta \circ \alpha \circ \sigma}(Y_i) \\
        \da<{c_{\alpha \circ \sigma}}           &               &  \\
        \dirlim (F_\beta \circ D^Y_k)           &               & \da>{c_{\beta \circ \alpha \circ \sigma}} \\
        \da<\iso                                &               &                                       \\
        F_\beta(L(Y)_k)                 & \lra_{l^\sharp_\beta} & L(Y)_l
\enddiagram
commutes, where the maps $c_{\alpha \circ \sigma}$ and $c_{\beta \circ
  \alpha \circ \sigma}$ are canonical maps to the colimit, and all
maps labelled with `$\iso$' are uniqueness isomorphisms. The upper
square commutes by uniqueness, and the lower square commutes by
definition of $l^\sharp_\beta$. This implies that the square~$(*)$
commutes, and Lemma~\ref{DiagramCond} shows that~$L(Y)$ is a twisted
diagram as claimed.

\begin{theorem}[Left \textsc{Kan\/} extensions]
  \label{LeftKan}        
  Let~$\B$ be a $\J$-bundle consisting of cocomplete categories,
  $\Phi\colon \I \ra \J$ a functor and~$Y$ a twisted diagram with
  coefficients in~$\Phi^*\B$. The assignment $Y \mapsto L(Y)$
  described above is the object function of a functor $ L\Phi \colon
  \Tw (\I, \Phi^* \B) \ra \Tw (\J, \B) $ which is left adjoint
  to~$\Phi^*$. 
\end{theorem}

\begin{proof}
  Abbreviate $L\Phi$ by~$L$ and keep the notation used in the
  construction of~$L(Y)$.

  We start by describing the effect of~$L$ on morphisms.  Let $f
  \colon Y \ra Z$ be a map of twisted diagrams with coefficients in
  $\Phi^*\B$, and fix an object $j \in \J$. For each $\sigma \colon
  \Phi(i) \ra j$, the maps $F_\sigma(f_i)$ form a natural
  transformation from $D^Y_j$ to $D^Z_j$, because the uniqueness
  isomorphisms are natural, $f$ is a map of twisted diagrams and
  $F_\sigma$ is a functor. This defines a map on the colimits $L(f)_j
  \colon L(Y)_j \ra L(Z)_j$.

        We claim that the maps~$L(f)_j$ assemble to a map~$L(f)$ of twisted diagrams.
        For $\alpha \colon j \ra k$ in $\J$, consider the diagram
\diagram
        F_\alpha(L(Y)_j)        & \lra^{F_\alpha(L(f)_j)}       & F_\alpha(Z_j)         \\
        \da^{l^\sharp_\alpha}   &                               & \da_{m^\sharp_\alpha} \\
        L(Y)_k                  & \lra_{L(f)_k}                 & L(Z)_k
\enddiagram
        where $l$ and $m$ denote the structure maps of $L(Y)$ and $L(Z)$. It commutes if and
        only if for each object $\sigma \colon \Phi(i) \ra j $ of~$\Phi \downarrow j$,
        the diagram
        \diagram
                F_\alpha\circ F_\sigma(Y_i)     &\lra^{F_\alpha\circ F_\sigma(f_i)}     &F_\alpha\circ F_\sigma(Z_i)\\
                \da^\iso                        &                                       & \da_\iso      \\
                F_{\alpha\circ\sigma}(Y_i)      & \lra_{F_{\alpha\circ\sigma}(f_i)}     & F_{\alpha\circ\sigma}(Z_i)\\
                \da                             &                                       & \da           \\
                L(Y)_k                          & \lra_{L(f)_k}                         & L(Z)_k
        \enddiagram
        commutes. The isomorphisms are uniqueness isomorphisms, which are natural, hence the
        upper square commutes. The lower vertical arrows denote the canonical map to the 
        colimit, and the naturality of these make the lower square commute. 

        Having checked that $L(f)$ is indeed a map of twisted
        diagrams, it is clear that $L$~is a functor, because maps of
        twisted diagrams are defined pointwise, and~$L_j$ is defined
        as the composition of functors~$\dirlim \circ \Psi^* \circ
        \pr_\I^*$, (with $\Psi$ and $\pr_I$ being explained below the
        definition of $D^Y_j$). To prove that $L$~is left adjoint
        to~$\Phi^*$, we construct natural transformations $\eta \colon
        \id \ra \Phi^* \circ L$ and $\epsilon \colon L \circ \Phi^*
        \ra \id$ satisfying the triangular identities \cite[\S IV.1,
        Theorem~2~(v)]{ML}.

  For $Y \in \Tw (\I,\Phi^*\B)$, the $Y$-component $\eta_Y$ is given
  (pointwise) as the canonical map to the colimit $Y_i \ra
  \Phi^*(L(Y))_i = L(Y)_{\Phi(i)}$ which corresponds to the identity
  $\id\colon\Phi(i) \ra \Phi(i)$ (an object of $\Phi \downarrow
  \Phi(i)$).  We check that $\eta_Y$~is a map of twisted diagrams.
  Let $\alpha\colon i \ra j \in \I$ be given and consider the diagram
  \begin{diagram}
    F_{\Phi(\alpha)}(Y_i)   & \lra^{F_{\Phi(\alpha)}((\eta_Y)_i)}   & F_{\Phi(\alpha)}(L(Y)_{\Phi(i)})\\
    \da<{y^\sharp_\alpha}   &                                       & \da>{l^\sharp_{\Phi(\alpha)}}   \\
    Y_j & \lra_{(\eta_Y)_j} & L(Y)_{\Phi(j)}
  \end{diagram}
  with the structure map $y^\sharp_\alpha$ starting from
  $G_\alpha(Y_i) = F_{\Phi(\alpha)}(Y_i)$ by definition of
  $\Phi^*\B$. Since the structure map $l^\sharp_{\Phi(\alpha)}$ is
  defined via the canonical maps to the colimit
  \[F_{\Phi(\alpha)}\circ F_\sigma(Y_k) \iso F_{\Phi(\alpha) \circ
    \sigma}(Y_k) \ra L(Y)_{\Phi(j)}\] 
  (for $\sigma\colon\Phi(k) \ra \Phi(i)$ an object of $\Phi \downarrow
  \Phi(i)$), the composition
  \[l^\sharp_{\Phi(\alpha)} \circ F_{\Phi(\alpha)}((\eta_Y)_i)\]
  coincides with the canonical map to the colimit
  \[c\colon F_{\Phi(\alpha)}(Y_i) \ra L(Y)_{\Phi(j)}\]
  (the special case $\sigma = id_{\Phi(i)}$). Hence we have to show
  that the triangle
  \begin{diagram}
    F_{\Phi(\alpha)}(Y_i)   &                       &               \\
    \da<{y^\sharp_\alpha}   & \rdTo^c                       &               \\
    Y_j & \lra_{(\eta_Y)_i} & L(Y)_{\Phi(j)}
  \end{diagram}
  commutes. But this is true by the definition of $L(Y)_{\Phi(j)}$ as
  the colimit of $D^Y_{\Phi(j)}$.  The naturality of $\eta$ can be
  explained as follows. For $i\in \I$, the canonical maps to the
  colimit $Y_j\ra L(Y)_{\Phi(i)}$ for varying $\sigma\colon j\ra
  \Phi(i)$ are a natural transformation of diagrams (with shape
  $\Phi\downarrow \Phi(i)$). In particular, the $\Phi(i)$-component,
  being the map $(\eta_Y)_i$, is natural.  We turn to the definition
  of~$\epsilon \colon L\circ \Phi^* \ra \id$.  For $Z \in \Tw
  (\J,\B)$, the map $\epsilon_Z$ is given pointwise as follows: for
  every $j \in \J$ and every $\sigma \colon \Phi(i) \ra j$ in $\Phi
  \downarrow j$, the structure maps $F_\sigma(\Phi^*(Z)_i) =
  F_\sigma(Z_{\Phi(i)}) \lra^{y^\sharp_\sigma} Z_j$ assemble to a
  natural transformation from~$D^{\Phi^* Z}_j$ to the constant diagram
  with value~$Z_j$ (this follows from Lemma~\ref{DiagramCond} and the
  fact that $Z$~is a twisted diagram). By the universal property of
  the colimit, this natural transformation defines a unique map $$
  (\epsilon_Z)_j \colon L(\Phi^*(Z))_j \ra Z_j \ . $$ To prove that
  $\epsilon_Z$ is a map of twisted diagrams, let $\alpha\colon j \ra k
  \in \J$ and consider the following diagram:
  \begin{diagram}[l>=5em]
    F_\alpha(L(\Phi^*(Z))_j)& \lra^{F_\alpha((\epsilon_Z)_j)}       & F_\alpha(Z_j)         \\
    \da<{m^\sharp_\alpha}   &                                       & \da>{z^\sharp_\alpha} \\
    L(\Phi^*(Z))_k & \lra_{(\epsilon_Z)_k} & Z_k
  \end{diagram}
  Using the universal property of the colimit, the definition of
  $\epsilon_Z$ and the definition of the structure map
  $m^\sharp_\alpha$, we are left to show that, for each
  $\sigma\colon\Phi(i) \ra j$, the diagram
  \begin{diagram}[l>=4em]
    F_\alpha(F_\sigma(Z_i)) & \lra^{F_\alpha(z^\sharp_\alpha)}      & F_\alpha(Z_j)         \\
    \da                     &                                       & \da>{z^\sharp_\alpha} \\
    L(\Phi^*(Z))_k & \lra_{(\epsilon_Z)_k} & Z_k
  \end{diagram}
  commutes, where the left vertical map is the composition of the
  uniqueness isomorphism and the canonical map to the colimit
  $F_{\alpha \circ \sigma}(Z_i) \ra L(\Phi^*(Z))_k$.  However, the
  definition of $\epsilon_Z$ implies that the diagram above commutes
  since~$Z$ is a twisted diagram.  To prove the naturality
  of~$\epsilon$, let $f\colon Y \ra Z$ be a map in $\Tw(\J,\B)$.  For
  $j\in \J$ and every $\sigma\colon \Phi(i)\ra j$ in $\Phi\downarrow
  j$, the maps $F_\sigma(f_\phi(i))\colon F_\sigma(Y_\phi(i))\ra
  F_\sigma(Z_\phi(i))$ assemble to a natural transformation
  $D^{\Phi^*f}_j$ of functors on $\Phi\downarrow j$ making the diagram
  \begin{diagram}
    D^{\Phi^*Y}_j           &       \rTo    &       Y_j             \\
    \dTo^{D^{\Phi^*f}_j}    &               &       \dTo_{f_j}      \\
    D^{\Phi^*Z}_j           &       \rTo    &       Z_j             \\
  \end{diagram}
  commute. The horizontal maps are the ones appearing in the
  definition of $\epsilon$.  Since the colimit functor is left adjoint
  to the ``constant diagram'' functor, the square
  \begin{diagram}
    L(\Phi^*(Y))_j          &       \rTo^{(\epsilon_Y)_j}   &       Y_j             \\
    \dTo^{L(\Phi^*f)}       &                               &       \dTo_{f_j}      \\
    L(\Phi^*(Z))_j          &       \rTo^{(\epsilon_Z)_j}   &       Z_j             \\
  \end{diagram}
  commutes, proving the naturality of $\epsilon$.  \medskip

  It remains to prove that the composites
  \[L \lra^{L\eta} L \circ \Phi^* \circ L \lra^{\epsilon L} L \qquad
  {\rm and} \qquad \Phi^* \lra^{\eta \Phi^*} \Phi^* \circ L \circ
  \Phi^* \lra^{\Phi^* \epsilon} \Phi^*\]
  are identity natural transformations. The verification is
  straightforward; we omit the details.
%
%
\end{proof}

The right adjoint of~$\Phi^*$, obtained by the corresponding twisted version of 
right \textsc{Kan\/} extension along~$\Phi$, will be denoted~$R\Phi$. By the dual
of Theorem~\ref{LeftKan} it exists if~$\B$ consists of complete categories.

Recall the functor $Ev_i$ defined as the restriction along $\{ i\} \ra \J$.
If the bundle $\B$ consists of cocomplete categories, its left adjoint 
\[Fr_i \colon \C_i \ra \Tw (\J,\B)\]
exists by Theorem~\ref{LeftKan}. It is the analogue of the free
diagram at $i$ and will be needed later in the construction of a
cofibrantly generated model structure. We call~$Fr_i (K)$ the {\it
  free twisted diagram generated by~$K \in \C_i$\/}.

\begin{example}[Spectra, continued]
  Let~$Sp$ be the bundle defined in Example~\ref{Spectra} which leads
  to ordinary spectra. The $n$th evaluation functor maps a spectrum to
  its $n$th term, and the corresponding $n$th free twisted diagram of
  a pointed simplicial set~$K$ is the spectrum
  \[* \rNonsense * \rNonsense \ldots \rNonsense * \rNonsense K
  \rNonsense \Sigma K \rNonsense \Sigma^2 K \rNonsense \ldots\]
  with~$K$ appearing at the $n$th spot and all $\sharp$-type structure
  maps being identities except for the map $\Sigma (*) = * \ra K$.
\end{example}

\subsection {Construction of Adjunction Bundles}
\label{subsec:construction}

We think of twisted diagrams as generalised diagrams. However, there
is an alternative approach using fibred and cofibred categories in the
sense of {\author Grothendieck\/}. For definitions and notation the
reader may wish to consult~\cite{Q1}.

Let us recall the \textsc{Grothendieck\/} construction~$\Gr (U)$ of a
contravariant functor~$U$ defined on~$\I$ with values in the category
of (small) categories.  The objects of~$\Gr (U)$ are the pairs~$(i,
Y)$ with~$i$ an object of~$\I$ and~$Y$ an object of~$U (i)$. A
morphism $(i, Y) \ra (j, Z)$ consists of a morphism $i \lra^\sigma j$
in~$\I$ and a morphism $Y \lra^A U(\sigma) (Z)$ in~$U(i)$. Composition
is given by the rule
\[(\tau, B) \circ (\sigma, A) := \bigl(\tau \circ \sigma, U(\sigma)
(B) \circ A \bigr) \ .\]
This construction comes equipped with a functor $\Gr (U) \ra \I$.

\begin{remark}
  \label{RemGroth}
  An adjunction bundle determines, by definition, a functor $U \colon
  \I^\op \ra Cat$, hence a functor $\Gr (U) \ra \I$.  The existence of
  the left adjoints~$F_\sigma$ make $\Gr (U)$ a cofibred category
  over~$\I^\op$, even a bifibred bundle in the sense of the next
  definition.
\end{remark}

\begin{definition}
  \label{Bundle}
  Given a functor $\pi \colon \mathcal{E} \ra \mathcal{A}$, we call~$\mathcal E$
  a {\it bifibred bundle over~$\mathcal A$\/} if the following conditions
  are satisfied (using notation from~\cite{Q1}):
  \begin{enumerate}
  \item The functor~$\pi$ is fibred, and for all composable morphisms
    $\alpha$ and~$\beta$ in~$\mathcal A$, the natural isomorphism
    $\alpha^* \circ \beta^* \ra (\beta \circ \alpha)^*$ is the
    identity.
  \item The functor~$\pi$ is cofibred, and for all morphisms $ \alpha
    \in \mathcal A $ the functor~$\alpha^*$ is right adjoint
    to~$\alpha_*$.
  \end{enumerate}
  In this situation, a functor $f \colon \I \ra \mathcal A$ determines an
  $\I$-indexed adjunction bundle $f \bowtie \pi = \I \bowtie_\mathcal{A}
  \mathcal{E}$ which sends the object $i \in \I$ to the category $
  \pi^{-1} (f (i)) $ and the morphism $\mu \in \I$ to the adjoint pair
  $f(\mu)_*$ and~$f(\mu)^*$.
\end{definition}

\begin{remark}[\textsc{M. Brun\/}'s reformulation of twisted diagrams]
  Recall from Remark~\ref{RemGroth} the functor $ \pi \colon \Gr (U) \ra
  \I $ associated to an adjunction bundle. A straightforward
  calculation which we omit shows that $ \Tw (\I, \B) $ is the
  category of sections of~$\pi$.

  More generally, given a bifibred bundle~$\pi$ and an adjunction
  bundle $f \bowtie \pi$ as in~\ref{Bundle}, the category of twisted
  diagrams $\Tw (\I, f \bowtie \pi)$ is the category of lifts of~$f$
  to~$\mathcal E$, \ie, the category of functors $g \colon \I \ra \mathcal E$
  satisfying $\pi \circ g = f$.
\end{remark}

\begin{example}
  \label{Toric}
  Let $\Mod \ra \Rng$ denote the canonical functor from the category
  of all modules over all rings to the category of rings. (The objects
  of~$\Mod$ are pairs~$(R, M)$ with~$R$ a ring and~$M$ an $R$-module.
  A morphism $(R, M) \ra (S, N)$ consists of a ring map $f \colon R
  \ra S$ and an $f$-semi-linear additive map $M \ra N$.)  This defines
  a bifibred bundle.

  A toric variety determines a functor into~$\Rng$, hence an
  adjunction bundle (cf.~Definition~\ref{Bundle}). In fact, the
  fan~$\Sigma$ of a toric variety can be regarded as a poset (ordered
  by inclusion of cones), hence as a category, and we obtain a functor
  \[\Sigma^\op \ra \Rng,\ \sigma \mapsto \cc [\check\sigma \cap M]\]
  where~$\check\sigma$ is the dual cone of~$\sigma$ and~$M$ is the
  dual lattice (see \textsc{Oda}~\cite{O} for details). Thus the toric
  variety~$X(\Sigma)$ determines the adjunction bundle $\Sigma^\op
  \bowtie_{\Rng} \Mod$.  As a more explicit example, the
  $n$-dimensional projective space is a toric variety. Its fan is
  isomorphic, as a poset, to the set of non-empty subsets of
  $[n]=\{0,1,\ldots,n\}$ (ordered by reverse inclusion).  The
  monoids~$\check\sigma\cap M=M^A$ are described in
  Example~\ref{ProjSpaces} below.
\end{example}

This example can be generalised to obtain an adjunction bundle from a
diagram of monoids and a cocomplete category~$\D$. We proceed with a
construction.
 
It is well known that we can consider any monoid~$M$ as a category
with one object and morphisms corresponding to the elements of~$M$. A
morphism of monoids then is a functor between two such categories.
Suppose that~$\D$ is a cocomplete category. We define the category of
$M$-equivariant objects in~$\D$, denoted $M \hbox {-} \D $, as the
category of functors $ M \ra \D $.  A monoid homomorphism $ f \colon M
\ra M^\prime $ induces the ``restriction'' functor $ f^* = R_f \colon
M^\prime \hbox {-} \D \ra M \hbox {-} \D $ (given by pre-composing
with~$f$). Since~$\D$ is cocomplete, this functor has a left adjoint $
f_* = \cdot \smashover {M} M^\prime \colon M \hbox {-} \D \ra M^\prime
\hbox {-} \D $.  For composable monoid homomorphisms we have the
relations $ (g \circ f)^* = f^* \circ g^* $ and $ (g \circ f)_* \iso
g_* \circ f_* $. Moreover $ \id^* = \id $, and we choose $ \id_* = \id
$.

Let~$Eq\D$ denote the category of equivariant objects in~$\D$. Objects
are the pairs $(M,\ D)$ where~$M$ is a monoid and~$D$ is a functor $M
\ra \D$. A morphism from $(M,\ D)$ to $(M^\prime,\ D^\prime)$ is a
pair~$(\alpha, \nu)$ where~$\alpha \colon M \ra M^\prime$ is a monoid
homomorphism and~$\nu$ is a natural transformation of functors $D \ra
D^\prime \circ \alpha$.  The forgetful functor $\pi \colon Eq\D \ra
Mon$ into the category of monoids make~$Eq\D$ into a bifibred bundle
in the sense of~\ref{Bundle}. The fibre over the monoid~$M$ is the
category~$M \hbox {-} \D$ of $M$-equivariant objects in~$\D$.

\begin{definition}
  \label{DefMonAdj}
  Suppose we have a (small) category~$\I$ and an $\I$-indexed
  diagram~$G$ of monoids, \ie, a functor $ G \colon \I \ra Mon $. For
  a cocomplete category~$\D$ we define the $\I $-indexed adjunction
  bundle $\mathrm{Ad}_\D G = (\C,\ F,\ U)$ by
  \[\mathrm {Ad}_\D G := \I \bowtie_{Mon} Eq\D \ .\]
  Explicitly, for an object~$i \in \I$ we let $\C_i := G(i) \hbox {-}
  \D$, the category of $G(i)$-equivariant objects in~$\D$, and for a
  morphism $\sigma \in \I$ we define $F_\sigma := G(\sigma)_*$ and
  $U_\sigma := G (\sigma)^*$.
\end{definition}

\bigskip

This definition is clearly natural in~$G$, \ie, given a natural
transformation of diagrams of monoids $G \ra G^\prime$ we obtain an
$\I$-morphism of adjunction bundles $\mathrm {Ad}_\D G^\prime \ra
\mathrm{Ad}_\D G$.

\begin{example}[Non-linear projective spaces]
  \label{ProjSpaces}
  This generalises the non-linear projective line
  (\ref{TheProjectiveLine}).  Let~$[n]$ denote the set $\{0,\ 1,\
  \ldots,\ n\}$, and write~$\langle n \rangle$ for the category of
  non-empty subsets of~$[n]$; morphisms are given by inclusion of
  sets. For $ A \subseteq [n] $, define the (additive) monoid
  \[M^A := \bigg\{(a_0,\ \ldots, a_n) \in \zz^{n+1} \,\Big|\, \sum_0^n
  a_i = 0 \hbox { and } \forall i \notin A \colon a_i \geq 0 \bigg\}\
  .\] These monoids assemble to a functor $D \colon \langle n \rangle
  \ra Mon$.  Let $Eq\DTop$ denote the category of equivariant spaces
  as constructed above. (Objects are pairs $(M, T)$ where $M$~is a
  monoid and $T$~is a pointed topological space with a base-point
  preserving \hbox{$M$-action}. Maps are semi-equivariant continuous
  maps of pointed topological spaces.) This category is a bifibred
  bundle over the category of monoids.  Thus we are in the situation
  of Definition~\ref{DefMonAdj} (with \hbox {$\I = \langle n
    \rangle$}); denote the resulting adjunction bundle
  $\mathrm{Ad}_{\Top} D = \langle n \rangle \bowtie_{Mon} Eq\DTop$ by
  $\mathfrak{P}^n$. The category of twisted diagrams $\Tw
  \bigl(\langle n\rangle, \mathfrak{P}^n\bigr)$ is nothing but the
  category~$\ppp^n$ of presheaves as defined in~\cite[6.1]{Hu1}.
\end{example}

\section{Model Structures}
\label{sec:model-structures}

\subsection{Bundles of Model Categories}

A {\it model category\/} is a category~$\C$ equipped with three classes of
distinguished morphisms, called weak equivalences, cofibrations and
fibrations. All these classes have to be closed under composition, and they
are required to contain the identity morphisms. This set of data is subject to
the following axioms:

\hangindent=8ex \hangafter 1 {\bf(MC1)} All finite limits and colimits
exist in~$\C$.

\hangindent=8ex \hangafter 1 {\bf(MC2)} If $f$ and $g$ are composable
morphisms, and if two of the three morphisms $f$, $g$ and~$g \circ f$
are weak equivalences, so is the third.

\hangindent=8ex \hangafter 1 {\bf(MC3)} The classes of weak
equivalences, cofibrations and fibrations are closed under retracts.

\hangindent=8ex \hangafter 1 {\bf(MC4)} Given a commutative square
diagram in $\C$
\begin{diagram}
  A & \rTo^f & X \\ \dTo<i && \dTo>p \\ B & \rTo_g & Y
\end{diagram}
where $i$ is a cofibration, $p$ is a fibration, and at least one
of~$i$ and~$p$ is a weak equivalence, there is a lift in the diagram,
\ie, a morphism $\ell \colon B \rTo X$ with $\ell \circ i = f$ and $p
\circ \ell = g$.

\hangindent=8ex \hangafter 1 {\bf(MC5)} Given any morphism $f$ there
is a factorisation $f = q \circ i$ where $i$ is a cofibration, and $q$
is a fibration and a weak equivalence.  Given any morphism $f$ there
is a factorisation $f = p \circ j$ where $j$ is a cofibration and a
weak equivalence, and $p$ is a fibration.

\medbreak

The term ``model category'' is always to be understood in the above
sense which is the definition given by \textsc{Dwyer} and
\textsc{Spalinski} \cite{DS}. It is slightly more general than the
definition given by \textsc{Hovey} \cite{Ho}, the differences being
the following: In \cite{Ho}, it is required that a model category has
all small limits and colimits (instead of just finite ones), and the
factorisations have to be functorial and are part of the structure
(instead of assuming that they simply exist).

\begin{definition}
  Let $\B = (\C,\ F,\ U)$ be an adjunction bundle over $\I$.  We
  call~$\B$ an {\it adjunction bundle of model categories\/} if all
  the~$\C_i$ are model categories, and all the~$F_\sigma$ preserve
  cofibrations and acyclic cofibrations. In other words, we require
  the pair $(F_\sigma ,U_\sigma)$ to form a {\author Quillen\/}
  adjoint pair.---If in addition all the~$\C_i$ are left proper model
  categories, $\B$~is called {\it left proper\/}, and similarly for
  ``right proper'' and ``proper''. Note that the inverse image of an
  adjunction bundle of model categories $\B$ is again an adjunction
  bundle of model categories, which is as proper as $\B$.
\end{definition}

\begin{example}
  The projective space bundles $ \mathfrak{P}^n$ (\ref{ProjSpaces})
  and spectra~$Sp$ (\ref{Spectra}) are examples of proper adjunction
  bundles of model categories. The model structure defined on $M
  \DTop$ (for~$M$ a monoid) has weak equivalences and fibrations on
  underlying spaces, the model structure on the category of pointed
  simplicial sets is the usual one.
\end{example}

Before defining the model structures on twisted diagrams, we make a
technical observation.

\begin{remark}
   \label{PatchingI}
   Suppose~$ \C = \prod_\nu \C_\nu $ is the product of model
   categories~$\C_\nu$.  Then there is a product model structure
   on~$\C$ where a map is a \weq{} (\resp, fibration, \resp,
   cofibration) if its image under the canonical projection is a
   \weq{} (\resp, fibration, \resp, cofibration) in~$\C_\nu$ for
   all~$\nu$ (see \cite[1.1.6]{Ho}).  If all the~$\C_\nu$ are left
   proper, $\C$ is a left proper model category, and similarly for
   ``right proper''.
\end{remark}

\subsection{The $c$-Structure}

The first model structure on~$\Tw (\I, \B) $ we want to consider has
pointwise \weqs{} and pointwise fibrations. The price one has to pay
for the simple definition of fibrations is that the description of
cofibrations is rather involved.  Moreover, we have to restrict to
``nice'' indexing categories.

\begin{definition}[Direct categories]
  \label{DefDirect}
  A {\it category with degree function\/} is a (small) category~$\I$
  together with a $\zz$-valued function~$d$, defined on the objects,
  such that whenever there is a non-identity morphism $ i \ra j $ we
  have $ d(i) \not= d(j) $. (We say that all non-identity arrows
  change the degree. In particular, objects have no non-trivial
  endomorphisms.)  The category is called {\it bounded\/} if~$d$ is
  bounded below, and it is called {\it locally bounded\/} if each
  connected component is bounded.  Without restriction, the degree of
  a bounded category has values in an honest ordinal, namely $\nn$.
  If non-identity arrows always increase the degree and the category
  is (locally) bounded, we say that~$\I$ is a {\it (locally) direct
    category\/}.
\end{definition}

All finite dimensional categories (\ie, categories with finite
dimensional nerve) admit degree functions and can be made into direct
categories. A disjoint union of locally direct categories is locally
direct.  If $\I$ is (locally) direct, so are subcategories, under and
over categories formed with~$\I$.  In particular, the full subcategory
$\I_n$ of objects of degree less than or equal to $n$ is (locally)
direct. A finite product of direct categories is direct (with degree
given by sum of partial degrees).

\medskip

In what follows, $\B = (\C, F, U)$ is an adjunction bundle of cocomplete model
categories over~$\I$.
Let~$Y$ be a twisted diagram with coefficients in~$\B$ and~$i$ an object of~$\I$. To describe
the cofibrations in the model structure we are going to construct, we have to introduce
the latching object of~$Y$ at~$i$. Recall that for a diagram~$Z$ (untwisted
case) the latching object at~$i$ is defined as the colimit over all components~$Z_j$
which map to~$Z_i$. For a twisted diagram~$Y$, we mimic this construction, using
the ``twisting'' functors~$F_\sigma$ to push everything into the category~$\C_i$.
The colimit is to be taken with respect to the $\sharp$-type structure maps of~$Y$.

Technically, we can describe the latching spaces as follows.
For each object~$ i \in \I $, let~$ \I \downarrow i $ denote the category of objects over~$i$. 
Let $\I \Downarrow i$ denote the full subcategory of $\I \downarrow i$ which consists of all
objects $\sigma \colon j \ra i$ with $\sigma \neq \id_i$. There are
canonical functors $\iota \colon \I \Downarrow i \rInto \I \downarrow i$ (the inclusion) and
$\pr \colon \I \downarrow i \ra \I$ (the projection $(\sigma\colon j\ra i) \mapsto j$). Set
$P_{\I \Downarrow i} := \pr \circ \iota$ and denote the trivial bundle over 
$\I\Downarrow i$ with value~$\C_i$ by~$\C_i$ again.
We define an $\I\Downarrow i$-morphism of bundles $\Psi \colon \C_i 
\ra (P_{\I \Downarrow i})^*\B$ as follows: For $\sigma \colon j \ra i$, the adjoint pair 
$$F_\sigma \colon \C_j \pile{\ra \\ \la} \C_i \colon U_\sigma$$
is the $\sigma$-component of $\Psi$, and it is obvious from the definitions that 
$\Psi$ is in fact a bundle morphism. Hence we have
a functor 
\[\Psi^* \colon \Tw (\I\Downarrow i,(P_{\I \Downarrow i})^*\B) \ra
\Fun(\I\Downarrow i,\C_i) \ .\]
Define $G_i \colon \Tw (\I\Downarrow i,(P_{\I \Downarrow i})^*\B) \ra
\C_i$ as the composition $G_i := \dirlim \circ \Psi^*$.

\begin{definition}
  \label{DefLatchObj}
  The {\it latching object~$L_i Y$ of~$Y$ at~$i$} is defined as $L_i Y
  := G_i\circ (P_{\I \Downarrow i})^* (Y)$. It is an object of $\C_i$.
\end{definition}

\begin{remark}
  \label{RemNatMap}
  Note that $L_i$ is a composition of functors, hence itself a
  functor. The structure maps $y^\sharp_\sigma \colon F_\sigma (Y_j)
  \ra Y_i$ for $\sigma \colon j \ra i$ define a natural transformation
  $ L_i \ra Ev_i$. If a map $L_i Y \ra Y_i$ is mentioned, it is always
  this natural map.
\end{remark}

\begin{example}
  If~$X$ is a spectrum and $n >0$, the latching object of~$X$ at~$n$
  is the pointed simplicial set $\Sigma X_{n-1}$, and the natural map
  $ \Sigma X_{n-1} \ra X_n $ of~\ref{RemNatMap} is the ($\sharp$-type)
  structure map of the spectrum.
\end{example}

\begin{example}
  Let $Y = (Y_+ \rNonsense^{y_\alpha} Y_0 \lNonsense^{y_\beta} Y_-) $
  be a twisted diagram with coefficients in the projective line bundle
  $\mathfrak{P}^1$ (cf.~\ref{TheProjectiveLine}).  The latching
  objects of~$Y$ at~$+$ and at~$-$ are the initial objects in $\nn_+
  \DTop$ and $\nn_- \DTop$, respectively. The latching object at~$0$
  is the $\zz$-equivariant pointed space $(Y_+\smashover {\nn_+}\zz)
  \vee (Y_-\smashover {\nn_-} \zz)$.  The $\sharp$-type structure maps
  induce a map to~$Y_0$.
\end{example}

\begin{definition}[The $c$-structure]
  Let $ f \colon Y \ra Z $ be a map in of twisted diagrams in $\Tw
  (\I,\B)$.  We call~$f$ a {\it weak equivalence\/} if~$f_i$ is a
  \weq{} in~$\C_i$ for every object~$ i \in \I$.  We call~$f$ a {\it
    $c$-cofibration\/} if for all objects~$i$ of~$I$, the induced map
  $\po Y_i, L_i Y, {L_i Z} \ra Z_i$ is a cofibration.  We call~$f$ a
  {\it $c$-fibration\/} if all~$f_i$ are fibrations in~$\C_i$.
\end{definition}

To prove that the $c$-structure is a model structure, we concentrate
on the lifting axiom first. Call a map $f \in \Tw (\I,\B)$ a {\it good
  acyclic $c$-cofibration\/} if for all objects~$i$ of~$I$, the
induced map $ \po Y_i, L_i Y, {L_i Z} \ra Z_i $ is an acyclic
cofibration.  Later, we will prove that the class of good acyclic
$c$-cofibrations coincides with the class of acyclic $c$-cofibrations.

\begin{lemma}
  \label{LiftingProperty}
  Let $\I$ be a direct category, and let $\B$ be an adjunction bundle
  of cocomplete model categories over $\I$.  Good acyclic
  $c$-cofibrations have the left lifting property with respect to
  $c$-fibrations. Similarly, $c$-cofibrations have the left lifting
  property with respect to acyclic $c$-fibrations.
\end{lemma}

\begin{proof}
  We treat the first case only, the other is similar. Let 
  \begin{diagram}
    A       &       \lra^g  &       X       \\
    \da<f   &               &       \da>p   \\
    B & \lra_h & Y
  \end{diagram}
  be a commutative diagram in $\Tw (\I,\B)$ such that $f$~is a good
  acyclic $c$-cofibration and $p$~is a $c$-fibration. We will
  construct the desired lift by induction on the degree of objects
  of~$\I$.
        
  Since $\I$ is direct, the degree function~$d$ has a
  minimum~$k$. If~$i$ is an object in~$\I$ of degree~$k$, then
  $L_i$~is the constant functor with the initial object as value. By
  definition of a good acyclic cofibration, the map~$f_i$ is an
  acyclic cofibration in $\C_i$. Hence we can find a lift~$l_i$ in the
  following diagram:
  \begin{diagram}
    A_i                     &       \lra^{g_i}      &       X_i             \\
    \dCof<{f_i}>\sim        &       \ruDotsto_{l_i} &       \dFib>{p_i}     \\
    B_i & \lra_{h_i} & Y_i
  \end{diagram}
  Since the full subcategory~$\I_k$ of objects of degree~$k$ is
  discrete, the lifts~$l_i$ for the various~$i \in \I_k$ assemble to a
  map $l|_{\I_k} \colon B|_{\I_k} \ra X|_{\I_k}$ in $\Tw
  (\I_k,\B|_{\I_k})$.
        
  Now let $n>k$, and assume that we have constructed a lift in the diagram
  \begin{diagram}
    A|_{\I_{n-1}}           & \lra^{g|_{\I_{n-1}}}    & X|_{\I_{n-1}}         \\
    \da<{f|_{\I_{n-1}}}     & \ruDotsto>{l|_{\I_{n-1}}} & \da>{p|_{\I_{n-1}}} \\
    B|_{\I_{n-1}}           & \lra_{h|_{\I_{n-1}}}    & Y|_{\I_{n-1}}         \\
  \end{diagram}
  making it a commutative diagram in $\Tw
  (\I|_{n-1},\B|_{\I_{n-1}})$. If $i$~is an object of degree~$n$ and
  $\sigma \colon j \ra i$ an object of~$\I \Downarrow i$, the map
  \[F_\sigma(B_j) \lra^{F_\sigma(l_j)} F_\sigma(X_j)
  \lra^{x^\sharp_\sigma} X_i\]
  is part of a natural transformation $\phi \colon L_iB \ra X_i$ such
  that the following square diagram commutes:
  \begin{diagram}
    L_i A           &       \ra             &       A_i             \\
    \da<{L_if}      &                       &       \da>{g_i}       \\
    L_i B & \lra_\phi & X_i
  \end{diagram}

  \noindent Hence we get another diagram 
  \begin{diagram}
    \po A_i, L_i A, {L_i B} &       \ra             &       X_i             \\
    \dCof<\sim              &                       &       \dFib_{p_i}     \\
    B_i & \lra_{h_i} & Y_i
  \end{diagram}
  in which, by hypothesis, the left vertical map is an acyclic
  cofibration and the right vertical map is a fibration. Thus a lift
  $l_i \colon B_i \ra X_i$ exists, and it is straightforward to check
  that these maps~$l_i$, together with the morphism $ l|_{\I_{n-1}}$,
  define a map of twisted diagrams $l|_{\I_n} \colon B|_{\I_n} \ra
  X|_{\I_n}$ such that the diagram
  \begin{diagram}
    A|_{\I_{n}}       & \lra^{g|_{\I_{n}}}  &       X|_{\I_{n}}       \\
    \da<{f|_{\I_{n}}} & \ruDotsto>{l|_{\I_{n}}}     &       \da>{p|_{\I_{n}}} \\
    B|_{\I_{n}}       & \lra_{h|_{\I_{n}}}  &       Y|_{\I_{n}}       \\
  \end{diagram}
  commutes. This completes the induction.
\end{proof}

Let $\Phi \colon \I \ra \J$ be a functor and $\A$ an adjunction bundle
of cocomplete model categories over $\J$. Obviously, the functor
\[\Phi^* \colon \Tw (\J,\A) \ra \Tw (\I,\Phi^*\A)\]
preserves weak equivalences and $c$-fibrations. The question is
whether $\Phi^*$ also preserves $c$-cofibrations. Under certain
conditions (which are satisfied in the case of interest) we can give a
positive answer.

\medbreak

Suppose the functor $\Phi \colon \I \ra \J$ is {\it injective at
  identities\/}, \ie, whenever~$\Phi (\sigma)$ is an identity
morphism, so is~$\sigma$.  (For example, a faithful functor is
injective at identities.)  Then $\Phi$~induces a functor
\[\Phi \Downarrow i \colon \ \I \Downarrow i \ra \J \Downarrow
\Phi(i)\]
which sends $\sigma \colon k \ra i$ to $\Phi(\sigma) \colon \Phi(k)
\ra \Phi(i)$.  This construction is compatible with the projection
functors:
\[\Phi \circ P_{\I \Downarrow i} = P_{\J
  \Downarrow \Phi(i)} \circ \Phi \Downarrow i \ .\]

Recall that a functor $F\colon \C \ra \D$ is called {\it final\/} if
for each $A \in \D$ the category $A\downarrow F$ of objects $F$-under
$A$ is non-empty and connected.

We say that the functor~$\Phi$ satisfies the {\it finality
  condition\/} if it is injective at identities, and the functor $\Phi
\Downarrow i$ is final for all objects~$i \in \I$.

\begin{lemma}
 \label{PhiLatching}
 Let $\Phi \colon \I \ra \J$ be a functor, $\B$ an adjunction bundle
 of cocomplete model categories over $\I$ and $i$ an object of
 $\I$. Denote by $L_i$ the $i$-th latching object functor of $\Tw
 (\I,\Phi^*B)$, and by $L^\prime_{\Phi(i)}$ the $\Phi(i)$-th latching
 object functor of $\Tw (\J,\B)$.  If $\Phi$ satisfies the finality
 condition, then there is a natural isomorphism $L_i \circ \Phi^* \iso
 L^\prime_{\Phi(i)}$.
\end{lemma}

\begin{proof}
  The functor $L_i$ is defined as the composition $\dirlim \circ \Psi^*
  \circ P_{\I \Downarrow i}^*$, with $\Psi$ being an $\I \Downarrow
  i$-morphism with $\sigma$-component given by the adjunction
  \[F_{\Phi(\sigma)} \colon C_{\Phi(j)} \pile{\ra \\ \la} \C_{\Phi(i)}
  \colon U_{\Phi(\sigma)}\]
  where $\sigma \colon j \ra i$ is an object of $\I \Downarrow i$. On
  the other hand, $L^\prime_{\Phi(i)}$ is the composition
  $L^\prime_{\Phi(i)} = \dirlim \circ \Theta^* \circ P_{J\Downarrow
    \Phi(i)}^*$, with $\Theta$ having the $\tau$-component given by
  the adjunction
  \[F_{\tau} \colon C_{j} \pile{\ra \\ \la} \C_{\Phi(i)} \colon
  U_{\tau}\]
  where $\tau \colon j \ra \Phi(i)$ is an object of $\J \Downarrow
  \Phi(i)$.  It is straightforward to check that the equality $L_i
  \circ \Phi^* = \dirlim \circ (\Phi \Downarrow i)^* \circ \Theta^*
  \circ P_{J\Downarrow \Phi(i)}^*$ holds. Hence the $i$-th latching
  object of $\Phi^*(A)$ is given by
  \[L_i (\Phi^* (A)) = \dirlim \circ (\Theta^* \circ P_{J\Downarrow
    \Phi(i)}^* (A)) \circ (\Phi \Downarrow i) \ .\]
  The functor~$\Phi \Downarrow i $ induces a map $L_i (\Phi^*(A)) \ra
  L^\prime_{\Phi(i)} (A)$ which is an isomorphism by \cite[IX.3.1]{ML}
  since $\Phi \Downarrow i$ is final.
\end{proof}

\begin{corollary}
 \label{FinalityCondition}
 If~$\Phi$ satisfies the finality condition, then $\Phi^*$~preserves
 $c$-cofibrations and good acyclic $c$-cofibrations.
\end{corollary}

\begin{proof}
  This follows from~\ref{PhiLatching} since the maps $L_i (\Phi^* A)
  \ra A_{\Phi(i)}$ and $L^\prime_{\Phi(i)} A \ra A_{\Phi(i)}$
  correspond under the isomorphism.
\end{proof}

\begin{remark}
  \label{Remark}
  The functor $P_{\I \Downarrow i}$ satisfies the finality condition
  because $(P_{\I \Downarrow i}) \Downarrow \alpha$ is an isomorphism
  of categories for each object $\alpha \in \I \Downarrow i$.
\end{remark}

\begin{lemma}
  \label{LatchingPreservesCofibrations}
  Let $\I$ be direct.  For each $i \in \I$, the latching object
  functor $L_i$ maps $c$-cofibrations to cofibrations and good acyclic
  $c$-cofibrations to acyclic cofibrations.
\end{lemma}

\begin{proof}
  Recall that $L_i$ was defined as the composite $G_i \circ (P_{\I
    \Downarrow i})^*$.  By Remark~\ref{Remark} and
  Corollary~\ref{FinalityCondition}, we are left to show that $G_i$ maps
  $c$-co\-fibra\-tions to cofibrations and good acyclic $c$-cofibrations
  to acyclic cofibrations. However, $G_i$ has a right adjoint $V_i :=
  \Psi_* \circ \delta $, where $\delta \colon \C_i \ra
  \Fun(\I\Downarrow i, \C_i)$ denotes the constant diagram functor and
  $\Psi_*$ is the direct image under the $\I\Downarrow i$-morphism
  $\Psi$ having $\sigma$-component
  \[F_\sigma \colon C_j \pile{\ra \\ \la} \C_i \colon U_\sigma\]
  where $\sigma \colon j \ra i$ is an object of $\I \Downarrow i$.  It
  is easy to see that $V_i$~maps (acyclic) fibrations to (acyclic)
  $c$-fibrations. Hence the statement follows from
  Lemma~\ref{LiftingProperty} and the fact that $\C_i$~is a model
  category.
\end{proof}

\begin{corollary}
  \label{LevelCofibrations}
  If~$f$ is a (good acyclic) $c$-cofibration, all its components are
  (acyclic) cofibrations in their respective categories. In
  particular, a good acyclic $c$-cofibration is an acyclic
  $c$-cofibration.
\end{corollary}

\begin{proof}
  Let $f \colon A \ra B$ be a $c$-cofibration. By
  Lemma~\ref{LatchingPreservesCofibrations}, the map $L_i f \colon L_i
  A \ra L_i B$ is a cofibration in $\C_i$, hence $A_i \ra
  A_i\cup_{L_iA}L_iB$ is a cofibration. Observe that $f_i$~factors as
  this last map followed by the map $A_i \cup_{L_i A} L_i B \ra B_i$. Since
  the latter is a cofibration by hypothesis, we conclude that~$f_i$ is
  a cofibration.---The other case is similar.
\end{proof}

\begin{theorem}
  \label{ThmCStruc}
  Suppose $\I$ is a locally direct category, and $\B$ is an adjunction
  bundle of cocomplete model categories over $\I$.
  \begin{enumerate}
  \item The $c$-structure is a model structure.
  \item A map $f$ of twisted diagrams is an acyclic $c$-cofibration if
    and only if for all objects $i \in \I$, the induced map $ \po Y_i,
    L_i Y, {L_i Z} \ra Z_i $ is an acyclic cofibration in $\C_i$.
  \item If $\mathfrak{B\/}$ is a left (\resp, right) proper bundle, the
    $c$-structure is left (\resp, right) proper.
  \end{enumerate}
\end{theorem}

\begin{proof}
  Let $ (\I_\nu) $ denote the family of path components of~$\I$. Then
  $ \I = \coprod \I_\nu $, and each of the~$ \I_\nu $ is a direct
  category.  Since $ \Tw (\I,\B) = \prod_\nu \Tw (\I_\nu,\B|_{\I_\nu})
  $, it is enough to show that the $c$-structure is a model structure
  for each of the categories $ \Tw (I_\nu,\B|_{\I_\nu}) $;
  by~\ref{PatchingI} we can equip $ \Tw (\I,\B) $ with the product model
  structure. Consequently, we can assume that $\I$~is direct.

  We use the axioms for model categories as given in~\cite{DS}.  First
  we note that the class of \weqs{} is closed under com\-po\-si\-tion
  since \weqs{} are defined pointwise. Similarly, the composition of
  two \hbox {$c$-fibrations} is a $c$-fibration again.

  Now assume we have two composable $c$-cofibrations $ A \lra^f B
  \lra^g C $.  To show that~$ g \circ f $ is a $c$-cofibration, we
  have to prove that for all objects $ i \in \I $ the induced map
  \[\po A_i, L_i A, {L_i C} \ra C_i\]
  is a cofibration in~$ \C_i $. But we can factor this map as
  \begin{eqnarray*}
    \po A_i, L_i A, {L_i C}
    &\lra^\iso& \po {\po A_i, L_i A, {L_i B}}, L_i B, {L_i C} \\
    &\lra[l>=3em]^x& \po B_i, L_i B, {L_i C} \\
    &\lra[l>=3em]^y& C_i
  \end{eqnarray*}
  where~$x$ is induced by~$f$, and~$y$ is induced by~$g$. But both of
  these maps are cofibrations (since they are cobase changes of
  cofibrations), hence so is their composite.

  It is obvious that each of the classes above contains all
  identities.

  Axiom~{\bf MC1}: existence of finite limits and colimits is
  guaranteed by~\ref{Completeness} since they exist in all~$ \C_i $.

  Axiom~{\bf MC2}: the ``2-of-3'' property for \weqs{} is satisfied
  since \weqs{} are defined pointwise and {\bf MC2} holds in all the
  categories~$ \C_i $.

  Axiom~{\bf MC3}: the class of \weqs{} is closed under retracts since
  \weqs{} are defined pointwise, and in each category~$ \C_i $ a
  retract of a \weq{} is a \weq{}.  Similarly, the class of fibrations
  is closed under retracts.
        
  Suppose~$ g \colon Y \ra Z $ is a retract of~$ f \colon A \ra B $
  and~$f$ is a $c$-cofibration. We have to show that for all objects $
  i \in \I_n $, the map $ \po L_i Z, L_i Y, {Y_i} \ra Z_i $ induced
  by~$g$ is a cofibration in~$ \C_i $.  But by functoriality of
  pushouts and latching objects, this map is a retract of the map $
  \po L_i B, L_i A, {A_i} \ra B_i $ induced by~$f$, which is a
  cofibration by hypothesis. Since~{\bf MC3} is valid in~$ \C_i $, the
  former map is a cofibration.  Hence~$g$ is a $c$-cofibration as
  claimed. This argument also shows that the class of good acyclic
  $c$-cofibrations is closed under retracts.

  Axiom~{\bf MC5}: let $f \colon A \ra X$ be a map in $\Tw
  (\I,\B)$. We will construct inductively a factorisation of~$f$ as a
  good acyclic $c$-cofibration followed by a $c$-fibration. (The other
  factorisation axiom is proved in a similar manner).  Let~$k$ be the
  minimum of the degree function on~$\I$, and let~$i$ be of
  degree~$k$. Then $f_i$ factors in $\C_i$ as $A_i \rCof^{g_i}_\sim
  T_i \rFib^{p_i} X_i$, with $g_i$ being an acyclic cofibration and
  $p_i$ being a fibration. The collection of these factorisations
  (where~$i$ ranges through all objects of degree~$k$) yields a
  factorisation of $f|_{\I_k}$ in $\Tw (\I_k,\B|_{\I_k})$ as
  $g|_{\I_k} \colon A|_{\I_k} \ra T|_{\I_k}$ followed by $p|_{\I_k}
  \colon T|_{\I_k} \ra X|_{\I_k}$.

  Let $n>k$, and assume we have already constructed a factorisation of
  $f|_{\I_{n-1}}$ in $\Tw (\I_{n-1},\B|_{\I_{n-1}})$ as the composite
  \[A|_{\I_{n-1}} \lra^{g|_{\I_{n-1}}} T|_{\I_{n-1}}
  \lra^{p|_{\I_{n-1}}} X|_{\I_{n-1}} \ .\]
  Let~$i$ be of degree~$n$.  The canonical functor $P_{\I\Downarrow i}
  \colon \I\Downarrow i \ra \I$ factors through the inclusion $\Phi
  \colon \I_{n-1} \rInto \I$ as $\Theta \colon \I \Downarrow i \ra
  \I_{n-1}$ since~$\I$ is direct. Recall the functor
  \[G_i \colon \Tw (\I\Downarrow i,(P_{\I\Downarrow i})^*\B) \ra
  \C_i\]
  appearing in the definition of the $i$-th latching object
  functor~$L_i$ (\ref{DefLatchObj}).  By definition, $L_i = G_i \circ
  P_{\I \Downarrow i} = G_i \circ \Theta^* \circ \Phi^*$, hence $G_i
  \circ \Theta^* (A|_{\I_{n-1}}) = L_i A$. The maps $F_\sigma (T_j)
  \lra^{F_\sigma (p_j)} F_\sigma (X_j) \lra^{x^\sharp_\sigma} X_i$ for
  the different objects $\sigma \colon j \ra i$ of $\I \Downarrow i$
  induce a map $G_i \circ \Theta^*(T|_{\I_{n-1}}) \ra X_i$ which makes
  the diagram
  \begin{diagram}[l>=3em]
    G_i \circ \Theta^*(A|_{\I_{n-1}})       &       =L_i A  &       \ra     &       A_i     \\
    \da<{G_i \circ \Theta^*(g|_{\I_{n-1}})} &               &               &       \da_{f_i}\\
    G_i \circ \Theta^* (T|_{\I_{n-1}}) & & \ra & X_i
  \end{diagram}
  commute. Now factor the induced map $A_i\cup_{L_i(A)}(G_i \circ
  \Theta^*)(T|_{\I_{n-1}}) \ra X_i$ as an acyclic cofibration $h_i
  \colon A_i \cup_{L_i(A)}(G_i \circ \Theta^*)(T|_{\I_{n-1}})
  \rCof^\sim T_i$ followed by a fibration $p_i \colon T_i \rFib X_i$
  in $\C_i$. The collection of the $T_i$ for the different objects $i$
  of degree $n$, together with $T|_{\I_{n-1}}$ define a twisted
  diagram in $\Tw (\I_n,\B|_{\I_n})$. The new structure maps for
  $\sigma \colon j \ra i$ are the compositions
  \[F_\sigma (T_j) \ra G_i \circ \Theta^*(T|_{\I_{n-1}}) \ra A_i
  \cup_{L_i(A)}(G_i \circ \Theta^*)(T|_{\I_{n-1}}) \rCof^{h_i}_\sim
  T_i\]
  where the first two maps are the canonical ones. If we define~$g_i$
  as the composition of the canonical map $A_i \ra A_i
  \cup_{L_i(A)}(G_i \circ \Theta^*)(T|_{\I_{n-1}})$ with $h_i$, it is
  straightforward to check that we get a factorisation $ f|_{\I_n} =
  p|_{\I_n} \circ g|_{\I_n} $ in $\Tw (\I_n,\B|_{\I_n}) $.  This
  completes the induction.
        
  We end up with a factorisation of~$f$ as $A \lra^g T \lra^p X$.  The
  object~$T|_{\I_n}$ we constructed in the induction step coincides
  with the restriction of~$T$, and similarly for the maps~$g$
  and~$p$. It is clear that~$p$ is a $c$-fibration in $\Tw (\I,\B)$.
  To complete the proof of axiom {\bf MC5}, it remains to show that
  the map~$g$ is a good acyclic $c$-cofibration. However, if~$i$ is of
  degree $k = \min d$, the map $A_i \cup_{L_i A} L_iT = A_i \lra^{g_i}
  T_i$ is an acyclic cofibration in~$\C_i$, and if~$i$ is of degree~$n
  > k$, the map $A_i \cup_{L_i A} L_iT \ra T_i$ coincides with the map
  $h_i \colon A_i \cup_{L_i(A)}(G_i \circ \Theta^*)(T|_{\I_{n-1}}) \ra
  T_i$ which is an acyclic cofibration in~$\C_i$. Hence~$g$ is a good
  acyclic $c$-cofibration.

  \medbreak

  We prove part~(2) of the theorem. We have already seen that every
  good acyclic $c$-cofibration is an acyclic $c$-cofibration
  (\ref{LevelCofibrations}).  To prove the converse, let $f \colon A \ra
  X$ be an acyclic $c$-cofibration. Factor~$f$ as a good acyclic
  $c$-cofibration $g \colon A \ra T$ followed by a $c$-fibration $p
  \colon T \ra X$, and note that~$p$ is an acyclic $c$-fibration by
  axiom~{\bf MC2}. The map~$f$ is in particular a $c$-cofibration, so
  we can find a lift in the diagram
  \begin{diagram}
    A       &       \lra^g          &       T       \\
    \da<f   &                       &       \da>p   \\
    X & \lra_{\id_X} & X
  \end{diagram}
  which expresses~$f$ as a retract of~$g$. Since good acyclic
  $c$-cofibrations are closed under retracts, we are done.

  \medbreak

  Knowing~(2), we see that axiom~{\bf MC4} is an immediate consequence
  of Lemma~\ref{LiftingProperty}. This finishes the proof of~(1).

  \medbreak

  Finally, recall from Proposition~\ref{Completeness} that pushouts and
  pullbacks are calculated pointwise. Since the components of a \weq{}
  ($c$-fibration, $c$-cofibration) are \weqs{} (fibrations,
  cofibrations) in the respective categories (use
  Corollary~\ref{LevelCofibrations} for the $c$-cofibrations),
  assertion~(3) follows.
\end{proof}

\begin{remark}
  The definition of a direct category can be extended to more general
  degree functions having ordinals as values, cf.~\cite{Ho}.
  The two inductive proofs of~\ref{LiftingProperty}
  and~\ref{ThmCStruc} can be completed with a discussion of the
  ``limit ordinal case'', thus giving the $c$-structure for a larger
  class of indexing categories.
\end{remark}

\subsection{The $f$-Structure}

The construction of the $c$-structure can be dualised. There is a
notion of a (locally) inverse category, and matching objects allow us
to define an $f$-structure with pointwise cofibrations and \weqs{}.
In the following, let $\B = (\C,\ F,\ U)$ be an adjunction bundle of
complete model categories over~$\I$. Denote by $i \Downarrow \I $ the
full subcategory of the under category $i \downarrow \I$ consisting of
objects $\sigma \colon i \ra j$ with $\sigma \neq \id_i$.  Again we
have a canonical functor $\Phi \colon i \Downarrow \I \ra
\I$. Consider~$\C_i$ as a trivial bundle over $i \Downarrow \I$, and
let $\Psi \colon \Phi^*\B \ra \C_i$ be the $i \Downarrow \I$-morphism
of bundles with $\sigma$-component given by the adjunction
\[F_\sigma \colon \C_i \pile{\ra \\ \la} \C_j \colon U_\sigma\]
for $\sigma \colon i \ra j$. Define $H_i \colon \Tw (i \Downarrow
\I,\Phi^*\B) \ra \C_i$ as the composition $\invlim \circ \Psi_*$. In
fact, $H_i$~coincides with the direct image functor~$\Xi_*$
where~$\Xi$ is the bundle morphism given by the pair $(\Psi,
i\Downarrow \I \ra \{i\})$ (here~$\{i\}$ is the trivial category).

\begin{definition}
  Let~$Y$ be a twisted diagram with coefficients in~$\B$.  The {\it
    matching object of~$Y$ at~$i$} is defined as $M_i Y := H_i \circ
  \Phi^* (Y)$.
\end{definition}

\begin{remark}
  The structure maps $y^\flat_\sigma \colon Y_i \ra U_\sigma (Y_j)$
  for $\sigma \colon i \ra j$ define a natural transformation $ Ev_i
  \ra M_i$. If a map $Y_i \ra M_i Y$ is mentioned, it is always this
  natural map.
\end{remark}

\begin{definition}[The $f$-structure]
  Let $ f \colon Y \ra Z $ be a map of twisted diagrams in $\Tw
  (\I,\B)$.  We call~$f$ a {\it \weq{}\/} if~$f_i$ is a \weq{}
  in~$\C_i$ for every object~$ i \in \I$.  We call~$f$ an {\it
    $f$-fibration\/} if for all objects~$i \in I$, the induced map $
  Y_i \ra \pb Z_i, M_i Z, {M_i Y} $ is a fibration.  We call~$f$ an
  {\it $f$-cofibration\/} if all~$f_i$ are cofibrations in~$\C_i$.
\end{definition}

\begin{definition}
  \label{DefInverse}
  A category with degree function is called a {\it (locally) inverse
    category\/} if its opposite category (with the same degree
  function) is (locally) direct (\ref{DefDirect}).
\end{definition}

\begin{theorem}
  Suppose $\I$ is a locally inverse category, and $\B$ is an
  adjunction bundle of complete model categories over $\I$.
  \begin{enumerate}
  \item The $f$-structure is a model structure.
  \item If~$f$ is an $f$-fibration, all its components are fibrations
    in their respective categories.
  \item A map $f \colon Y \ra Z $ of twisted diagrams is an acyclic
    $f$-fibration if and only if for all objects $ i \in \I $, the
    induced map
    \[Y_i \ra \pb Z_i, M_i Z, {M_i Y}\]
    is an acyclic fibration in~$\C_i$.
  \item If~$\B$ is a left \resp right proper bundle, the $f$-structure
    is left \resp right proper.
  \end{enumerate}
  \qed
\end{theorem}

\begin{remark}
  In fact, it is possible to construct a model structure on the
  category~$\Tw (\I,\B)$ if~$\I$ is a {\author Reedy\/} category and
  $\B$~consists of complete and cocomplete model categories. One has
  to combine the construction of the $c$-structure and the
  $f$-structure. The \weqs{} are pointwise \weqs{}, the fibrations and
  cofibrations are more complicated to define. In the case of diagram
  categories, this is done in section 5.2 of \cite{Ho}, and the proof
  given there applies to our situation as well.
\end{remark}

\subsection {The $g$-Structure}

In this section we consider a cofibrantly generated model
structure with pointwise \weqs{} and pointwise fibrations. (In
particular, the $g$-structure coincides with the $c$-structure
provided both are defined.)  Terminology is taken from~\cite{Ho}.

\begin{definition}
  An $\I$-bundle~$\B$ of cocomplete model categories is called a {\it
    cofibrantly generated adjunction bundle\/} if for all objects $ i
  \in \I $ the model category~$\C_i$ is cofibrantly generated.
\end{definition}

Examples of cofibrantly generated adjunction bundles include the
spectrum bundle~$Sp$ of Example~\ref{Spectra} and the projective space
bundle~$\mathfrak{P}^n$ of~\ref{ProjSpaces}.  The inverse image of a
cofibrantly generated adjunction bundle is cofibrantly generated.

\medbreak

Since~$\C_i$ is cocomplete, the $i$-th evaluation functor $Ev_i
\colon \Tw (\I,\B) \ra \C_i$ has a left adjoint $Fr_i \colon \C_i \ra
\Tw (\I,\B)$, the $i$-th free twisted diagram functor obtained by
twisted left {\author Kan} extension (Theorem~\ref{LeftKan}). Explicitly,
for an object~$A$ of~$\C_i$ the $j$-component of~$Fr_i(A)$ is given by
the coproduct
\[\coprod_{\alpha \in \hom_\I (i, j)} F_\alpha (A)\]
and the structure maps are given in the following way: if $\beta
\colon j \ra k $ is a morphism in~$ \I $, the map
$Fr_i(A)^\sharp_\beta$ is the composition
\begin{eqnarray*}
  F_\beta \bigl(Fr_i(A)_j\bigr) &= & F_\beta \Bigl(\coprod_{\alpha \in
    \hom_\I (i, j)} F_\alpha (A)\Bigr) \\
  &\iso& \coprod_{\alpha \in \hom_\I (i, j)} F_\beta \circ F_\alpha
  (A) \\
  &\iso & \coprod_{\alpha \in \hom_\I (i, j)} F_{\beta
    \circ \alpha} (A) \\
  &\ra& \coprod_{\gamma \in \hom_\I (i, k)} F_\gamma (A)
\end{eqnarray*}
where the last map is the canonical map induced by the identity on
each summand, mapping the $\alpha$-summand of the source into the
$\beta \circ \alpha$-summand of the target. 

Define~$M$ to be the set of maps in $\Tw (\I,\B)$ of the form~$Fr_i
(f)$ with~$i$ some object of~$\I$ and~$f$ a generating cofibration
in~$\C_i$. Define~$N$ to be the set of maps in $\Tw (\I,\B)$ of the
form~$Fr_i (f)$, with~$i$ some object of~$\I$ and~$f$ a generating
acyclic cofibration in~$\C_i$. Note that~$M$ and~$N$ are sets
because~$\I$ is small.

\begin{definition}[The $g$-structure]
  Let $f \colon Y \ra Z$ be a map of twisted diagrams in $\Tw
  (\I,\B)$. We call $f$ a {\it weak equivalence\/} if $f_i$ is a weak
  equivalence in $\C_i$ for every object $i \in \I$. We call $f$ a
  {\it $g$-fibration\/} if $f$ has the right lifting property with
  respect to the set $N$.  We call $f$ a {\it $g$-cofibration\/} if
  $f$ has the left lifting property with respect to every
  $g$-fibration which is also a weak equivalence.
\end{definition}

\begin{lemma}
  \label{RightLiftingProperty}
  A map has the right lifting property with respect to the set $N$
  (\resp~$M$) if and only if all its components are fibrations
  (\resp~acyclic fibrations).
\end{lemma}

\begin{proof}
  This follows from the adjointness of $Fr_i$ and $Ev_i$, and the fact
  that $\B$ is cofibrantly generated.
\end{proof}

\begin{lemma}
  \label{SmallObjectArgument}
  The domains of the maps of $M$ are small relative to $M$-cell. The
  domains of the maps of $N$ are small relative to $N$-cell.
\end{lemma}

\begin{proof}
  This follows from the adjointness of $Fr_i$ and $Ev_i$, and the fact
  that $\B$ is cofibrantly generated. We give a detailed argument for
  the case of $M$.  Let $A$ be the domain of a map in $M$, so $A$ is
  of the form $Fr_i(X)$ for some $i \in \I$, with $X$ being the domain
  of a generating cofibration in $\C_i$.  Denote the set of generating
  cofibrations in $\C_i$ by $J$ and recall that $X$ is $\kappa$-small
  relative to the class $J$-cell for some cardinal $\kappa$, because
  $\C_i$ is cofibrantly generated. We will prove that $A = Fr_i(X)$ is
  $\kappa$-small relative to the class $M$-cell.

  Let $\lambda$ be a $\kappa$-filtered ordinal and $B \colon \lambda
  \ra \Tw (\I,\B)$ be a functor such that the map $B_\beta \ra
  B_{\beta +1}$ is in $M$-cell for all $\beta$ with $\beta+1 <
  \lambda$.  We have to prove that the canonical map
  \[\dirlim \Tw (\I,\B)(A,B_\beta) \ra \Tw (\I,\B)(A, \dirlim B)\]
  is an isomorphism. The adjointness of~$Fr_i$ and~$Ev_i$ provides
  that this map is isomorphic to the composite
  \[\dirlim \C_i(X,Ev_i(B_\beta)) \ra \C_i (X, Ev_i \circ \dirlim B)
  \iso \C_i(X, \dirlim Ev_i\circ B)\]
  (where the isomorphism is the one from
  Proposition~\ref{Completeness}). This composite is the canonical
  map, and $X$ is $\kappa$-small relative to $J$-cell. By
  \cite[2.1.16]{Ho}, $X$ is then even $\kappa$-small relative to the
  class of cofibrations in $\C_i$. Hence we are done if for
  all~$\beta$ with \hbox {$ \beta +1 < \lambda $} the map \hbox {$
    Ev_i(b) \colon Ev_i (B_\beta) \ra Ev_i(B_{\beta +1})$} is a
  cofibration. However, since the maps in $M$ are in particular
  pointwise cofibrations, and the class of pointwise cofibrations is
  closed under cobase changes and transfinite compositions, every map
  in $M$-cell is a pointwise cofibration. This finishes the proof.
\end{proof}

\begin{theorem}
  Let~$\B$ be a cofibrantly generated bundle over~$\I$.  The
  $g$-structure is a model structure on $\Tw (\I,\B)$ which is
  cofibrantly generated by the sets~$M$ and~$N$.
\end{theorem}

\begin{proof}
  We use Theorem~2.1.19 of~\cite{Ho}, which applies also for model
  categories in the sense of~\cite{DS}. The weak equivalences clearly
  define a subcategory which is closed under retracts and satisfies
  {\bf MC2}, so condition~1 holds.  Lemma~\ref{SmallObjectArgument}
  implies conditions~2 and~3, and Lemma~\ref{RightLiftingProperty}
  implies conditions~5 and~6, and one half of condition~4. It remains
  to prove that every map in $N$-cell is a weak equivalence.  Since
  every map in~$N$ is pointwise an acyclic cofibration, and the class
  of pointwise acyclic cofibrations is closed under pushouts and
  transfinite compositions, every map in $N$-cell is pointwise an
  acyclic cofibration, so in particular a weak equivalence.
\end{proof}

\begin{remark}
  From the general theory of cofibrantly generated model structures,
  we know that {\it a morphism~$f$ of twisted diagrams is a
    $g$-cofibration if and only if it is a retract of a transfinite
    composition of cobase changes of maps in~$M$\/}. Similarly,
  acyclic $g$-cofibrations can be characterised using the set~$N$.
\end{remark}

\section{Sheaves and Homotopy Sheaves}
\label{sec:sheaves}

Let~$\C$ be a model category, and suppose the diagram category $\Fun
(\I, \C)$ carries a model structure with pointwise \weqs{} as
described in one of the previous sections. There is a canonical
functor
\[h \colon \Fun (\I,\ \C) \ra \Fun (\I,\ \Ho \C)\]
which replaces each structure map of a diagram by its homotopy class
(or, more precisely, by its image under the localisation functor $ \C
\ra \Ho \C$).  In particular, the structure maps of a diagram~$Y$ are
\weqs{} if and only if the structure maps of~$h(Y)$ are isomorphisms.

The functor~$h$ factors through a functor
\[\hbar \colon \Ho \Fun (\I,\ \C) \ra \Fun (\I,\ \Ho \C)\]
which is, in general, {\it not\/} an equivalence of categories.

In this section, we construct such functors~$h$ and~$\hbar$ for
twisted diagrams. Unfortunately, this is not as straightforward as one
could expect since formation of total derived functors is not
functorial. One way to explain this is the following: if the
functor~$U$ preserves \weqs{} between fibrant objects, its total right
derived~$\RR U$ exists, and~$\RR U (Y)$ is given by evaluating~$U$ on
a fibrant replacement of~$Y$. Thus for $ U = \id_{\C} $ we see that
if~$\C$ contains objects which are not fibrant, {\it the total right
  derived of the identity functor~$\id_{\C}$ is isomorphic to, but
  different from, the functor~$\id_{\Ho \C}$.}

We remedy this by focusing on the full subcategory~$\C^f$ of fibrant
objects in~$\C$.  This is possible since by a theorem of
\textsc{Quillen\/}, the localisation of~$\C^f$ with respect to \weqs{}
is equivalent to the homotopy category of~$\C$.

\subsection{Associated Homotopy Bundle}

\label{AssocSec}

Let~$\C$ denote a model category, and denote by~$\C^f$ the full
subcategory of fibrant objects. The homotopy category~$\Ho \C$ is the
localisation of~$\C$ with respect to the class of \weqs{}. We will use
the rather explicit model described in~\cite{DS}: the objects of~$\Ho \C$
are the objects of~$\C$, and morphisms are homotopy classes of maps
between cofibrant-fibrant replacements.  Let~$\Hof \C$ denote the full
subcategory of~$\Ho \C$ generated by the fibrant objects.

\begin{lemma}
  \label{LocFib}
  The following diagram commutes:
  \begin{diagram}
    \C^f & \ra & \C \\ \da<{\gamma_f} && \da>{\gamma} \\ \Hof \C & \ra & \Ho \C \\
  \end{diagram}
  The vertical arrows are localisations with respect to the class of
  \weqs{}, the horizontal arrows are full embeddings. The lower
  horizontal arrow is an equivalence of categories.
\end{lemma}

\begin{proof}
  This follows from \textsc{Quillen}'s theorem on existence of
  homotopy categories \cite[\S I.1, Theorem~1]{Q2}. We omit the details.
\end{proof}



Now choose, for each object~$X \in \C$, a cofibrant replacement
\[p_X \colon X^c \rFib^{\sim} X \ .\]
If~$X$ is cofibrant itself, we choose $p_X = \id_X$. Similarly, we
choose fibrant replacements
\[q_X \colon X \rCof^{\sim} X^f\]
with $q_X = \id_X$ for fibrant~$X$.---The following Proposition is a
standard exercise in model category theory:

\begin{proposition}
  \label{Omnibus}
  Suppose $ U \colon \C \ra \D $ is right {\author Quillen\/} with left adjoint~$F$.
  \begin{enumerate}
  \item The total right derived $\RR U \colon \Ho \C \ra \Ho \D $
    exists and is given by $ \RR U(X) := U (X^f) $ on objects.
    Moreover, the functor~$\RR U$ has a left adjoint~$\LL F$.
  \item The image of the functor~$\RR U$ lies inside~$\Hof \D$,
    hence~$\RR U$ induces (by restriction) a functor $ \RR_f U \colon
    \Hof \C \ra \Hof \D $.
  \item Every map $ \alpha \colon X \ra Y $ in~$\Hof \C$ is
    represented by a diagram of the form $X \lFib^{\sim} X^c \lra^{f}
    Y$ in~$\C$.
  \item The functor~$\RR_f U$ is given by the identity on objects
    and, using the description of~(3), by $\RR_f U(\alpha) =
    \gamma_f (U (f)) \circ (\gamma_f (U(p_X)))^{-1}$ on morphisms.
  \item We have $\RR_f U \circ \gamma_f^{\C} = \gamma_f^{\D} \circ U$.
    Moreover, the functor~$\RR_f U$ is a left {\author Kan\/}
    extension of~$U$ along~$\gamma_f$.
  \item The equalities $\RR_f \id_{\C^f} = \id_{\Hof \C}$ and $\RR_f
    (V \circ U) = \RR_f V \circ \RR_f U$ hold.
  \item The functor~$\RR_f U$ has a left adjoint, denoted~$\LL_f F$,
    given (on objects) by the formula $\LL_f F(X) := F (X^c)^f$. \qed
  \end{enumerate}
\end{proposition}


In view of the previous lemma, parts~(2) and~(5) mean that $\RR_f
U$~is a ``good'' substitute for~$\RR U$. Moreover, by parts~(6)
and~(7), the following definition makes sense:

\begin{definition}[Associated homotopy bundle]
  If $\B = (\C,\ F,\ U)$ is an $\I$-indexed adjunction bundle of model
  categories, we define its {\it associated homotopy bundle of fibrant
    objects\/} $\Hof \B = (\Hof \C,\ \LL_f F,\ \RR_f U)$ as the
  $\I$-indexed adjunction bundle given by $i \mapsto \Hof \C_i$ for
  objects $i \in \I $ and $ \sigma \mapsto \LL_f F_\sigma$ and $\sigma
  \mapsto \RR_f U_\sigma$ for morphisms $ \sigma \in \I $.
\end{definition}

\subsection{Construction of~$h$ and $\hbar$}
\label{subsec:construction_of_h}

Suppose $\B = (\C,\ F,\ U)$ is an $\I$-indexed adjunction bundle of
model categories. We assume that we can equip~$\Tw (\I, \B)$ with a
model structure with pointwise \weqs{} (this is certainly possible if
$\I$~is locally direct or locally inverse, or if $\B$~is a cofibrantly
generated bundle). We want to associate to each twisted diagram $Y \in
\Tw (\I, \B)$ a corresponding twisted diagram $h(Y) \in \Tw (\I, \Hof
\B)$.

Assume for the moment that~$Y$ is a twisted diagram with fibrant
components.  Let~$Z$ denote the following twisted diagram:
\begin{gather*}
  i \mapsto \gamma_f (Y_i) = Y_i\\
  \sigma \mapsto \gamma_f (y_\sigma^\flat) \colon Y_i \ra U_\sigma
  (Y_j) = \RR_f U_\sigma (\gamma_f (Y_j))
\end{gather*}
We need to check the commutativity condition: if $i \lra^\sigma j
\lra^\tau k$ are composable morphisms in~$\I$, the following diagram
is supposed to commute:
\begin{diagram}
  \gamma_f (Y_i) & \lra^{\gamma_f (y_\sigma^\flat)} & \RR_f U_\sigma
  (\gamma_f (Y_j)) \\
  \dTo<{\gamma_f (y_{\tau\sigma}^\flat)} && \dTo>{\RR_f U_\sigma
    (\gamma_f (y_\tau^\flat))} \\ 
  \RR_f U_{\tau\sigma} (\gamma_f (Y_k)) & \rEqual & \RR_f U_\sigma
  \circ \RR_f U_\tau (\gamma_f (Y_k)) \\
\end{diagram}
Using~\ref{Omnibus}~(5) we see that this is just the corresponding diagram
for~$Y$ after application of~$\gamma_f$, hence commutes as desired.

A morphism $f \colon Y \ra \bar Y$ between pointwise fibrant twisted
diagrams induces a map $ g \colon Z \ra \bar Z $ with components $g_i
= \gamma_f (f_i)$ as can be shown using~\ref{Omnibus}~(5) and
functoriality of~$\gamma_f$.

\medbreak

Now we use this construction to define the actual functor~$h$
(or~$\hbar$). We discuss three cases in order of increasing
difficulty.

\medbreak

{\bf Case~1:} All the model categories~$\C_i$ used in the bundle~$\B$
consist of fibrant objects only.  Then $\C_i^f = \C_i$ and $\RR_f
U_\sigma = \RR U_\sigma$.  The assignment $Y \mapsto Z$ defines (the
object function of) a functor~$h$. By construction it maps \weqs{} to
isomorphisms, hence descends to a functor $\hbar \colon \Ho \Tw (\I,
\B) \ra \Tw (\I, \Hof \B)$.

\medbreak

{\bf Case~2:} Suppose that fibrant objects of~$\Tw (\I, \B)$ are
pointwise fibrant.  Suppose moreover that~$\Tw (\I, \B)$ has a fibrant
replacement functor $Y \mapsto Y^f$.  Then we can apply the above
construction to~$Y^f$ instead of~$Y$, and the composite $Y \mapsto Y^f
\mapsto Z$ defines (the object function of) a functor~$h$.  By
construction it maps \weqs{} to isomorphisms, hence descends to a
functor $\hbar \colon \Ho \Tw (\I, \B) \ra \Tw (\I, \Hof \B)$.

\medbreak

{\bf Case~3:} Suppose that fibrant objects of~$\Tw (\I, \B)$ are
pointwise fibrant.  Let~$\mathfrak{K}$ denote the category with
objects the fibrant and cofibrant twisted diagrams in~$\Tw (\I, \B)$,
and morphisms the homotopy classes of maps between such
objects. By~\cite[5.6]{DS} the inclusion $\nu \colon \mathfrak{K} \ra
\Ho \Tw (\I, \B)$ is an equivalence of categories.  Thus it suffices
to construct a functor $\phi \colon \mathfrak{K} \ra \Tw (\I, \Hof
\B)$; then we can define~$\hbar$ by the composition of an inverse
of~$\nu$ with~$\phi$.

An object $Y \in \mathfrak{K}$ is in particular a pointwise fibrant
twisted diagram. Hence the construction preceding case~1 applies, and
we can define $\phi (Y):= Z$.

A morphism $f \colon Y \ra \bar Y$ in~$\mathfrak{K}$ can be
represented by a map $\bar f \colon Y \ra \bar Y$ in~$\Tw (\I, \B)$
by~\cite[5.7]{DS}, and~$\bar f$ induces $\phi (f) \colon
\phi (Y) \ra \phi (\bar Y)$ with components $\phi (f)_i = \gamma_f
\bar f_i$.  To show that~$\phi (f)$ does not depend on the choice
of~$\bar f$, recall that homotopy is an equivalence relation for maps
$Y \ra \bar Y$ by~\cite[4.22]{DS}. Moreover, the evaluation
functors~$Ev_i$ (given by $Y \rMapsto Y_i$) commute with products and
preserve \weqs{}. Hence they preserve path objects and right
homotopies. Thus if~$\bar f$ and~$\bar g$ are homotopic, so are~$\bar
f_i$ and~$\bar g_i$. Since the localisation functor identifies
homotopic maps, this proves that~$\phi (f)$ is well defined.

Since homotopy is compatible with composition \cite[4.11 and
4.19]{DS}, and since the identity morphisms in~$\mathfrak{K}$ are
represented by identity maps, $\phi$~is a functor as required.

\subsection {Comparison of Sheaves and Homotopy Sheaves}

\begin{definition}[Left strict sheaves]
  Given an $\I$-indexed adjunction bundle~$\B$, we call an object $Y
  \in \Tw (\I, \B)$ a {\it left strict sheaf\/} if the $\sharp$-type
  structure map $y_\sigma^\sharp \colon F_\sigma (Y_i) \ra Y_j$ is an
  isomorphism for all morphisms $\sigma \colon i \ra j$ of~$ \I $.  We
  write~$\shv (\I, \B)$ for the full subcategory of~$\Tw (\I, \B)$
  generated by left strict sheaves.
\end{definition}

There is also a dual notion of a {\it right strict sheaf\/} requiring
that all $\flat$-type structure maps are isomorphisms.

\begin{example}[Quasi-coherent sheaves on toric varieties]
  \label{example:quasi_coh_toric}
  Recall the adjunction bundle $\Sigma^\op \bowtie_{\Rng} \Mod$
  associated to a toric variety~$X$ with fan~$\Sigma$,
  cf.~\ref{Toric}.  We claim that the category $\shv (\Sigma^\op,
  \Sigma^\op \bowtie_{\Rng} \Mod)$ is equivalent to the category of
  quasi-coherent sheaves on~$X$.  To see this, recall that a
  cone~$\sigma \in \Sigma$ corresponds to an open affine
  sub-scheme~$U_\sigma$ of~$X$. Given a quasi-coherent
  sheaf~$\mathcal{F}$, the associated twisted diagram is given by
  $\sigma \mapsto \mathcal{F} (U_\sigma)$ with $\flat$-type structure
  maps given by restriction maps.  Conversely, a left strict sheaf~$Y$
  defines quasi-coherent sheaves~$\tilde Y_\sigma$ on the
  sub-schemes~$U_\sigma$ which can be glued via the $\sharp$-type
  structure maps to give a quasi-coherent sheaf on~$X$. The details
  are left to the reader.
\end{example}

\begin{definition}[Left homotopy sheaves]
  \label{def:homotopy_sheaves}
  Suppose that~$\B$ is an adjunction bundle of model categories.  We
  call an object $Y \in \Tw (\I, \B) $ a {\it left homotopy sheaf\/}
  if for all morphisms $ \sigma \colon i \ra j $ of~$ \I $ there is an
  acyclic fibration $ \bar Y_i \rFib^\sim Y_i $ in~$\C_i$ with~$\bar
  Y_i$ cofibrant such that the adjoint to the composite
  \[\bar Y_i \rFib^\sim Y_i \lra^{y_\sigma^\flat} U_\sigma (Y_j)\]
  is a \weq{} in~$\C_j$. We write $\hshv (\I, \B)$ for the full
  subcategory of~$\Tw (\I, \B)$ generated by left homotopy sheaves.
\end{definition}

\begin{theorem}[Comparison of strict sheaves and homotopy sheaves]
  \label{HomInv}
  Let~$\B$ denote an $\I$-indexed adjunction bundles of model
  categories.  Assume that we have a map~$\hbar$ as given by one of
  the cases of~\S\ref{subsec:construction_of_h}.  An object $ Y \in
  \Tw (\I, \B) $ is a left homotopy sheaf if and only if $ \hbar (Y)
  \in \Tw (\I, \Hof \B) $ is a left strict sheaf. In particular, if $
  Y \lra^\sim Z $ is a \weq{} of twisted diagrams, $Y$~is a left
  homotopy sheaf if and only if~$Z$ is.
\end{theorem}

\begin{proof}
  Fix a morphism $\sigma \colon i \ra j $ of~$ \I$, and define $Z:=
  \hbar (Y)$.  By construction, $z_\sigma^\flat$ is a morphism
  in~$\Hof \C_i$ which is isomorphic, in~$\Ho \C_i$, to a morphism
  $k^\flat \colon Y_i \ra \RR U_\sigma (Y_j)$. The isomorphism is
  given by the fibrant replacement used in the construction
  of~$\hbar$. If~$Y_i$ and~$Y_j$ happen to be fibrant, the
  maps~$z_\sigma^\flat$ and~$k^\flat$ agree.

  There is a commutative diagram of categories and functors
  \begin{diagram}
    \Hof \C_i & \lla^{\RR_f U_\sigma} & \Hof \C_j \\
    \da && \da \\
    \Ho C_i & \lla_{\RR U_\sigma} & \Ho \C_j \\
  \end{diagram}
  where both vertical arrows are equivalences.
  Hence~$z_\sigma^\sharp$ is isomorphic, in the category~$\Ho \C_j$,
  to the adjoint $k^\sharp \colon \LL F_\sigma (Y_i) \ra Y_j$
  of~$k^\flat$.  In particular, the morphism~$z_\sigma^\sharp$ is an
  isomorphism if and only if~$k^\sharp$ is.

  Choose a cofibrant replacement $q_i \colon Y_i^c \rFib^\sim Y_i$
  of~$Y_i$ and a fibrant replacement $p_j \colon Y_j \rCof^\sim Y_j^f$
  of~$Y_j$.  Let~$\ell^\flat$ denote the composite map
  \[Y_i^c \rFib^\sim_{q_i} Y_i \lra^{y_\sigma^\flat} U_\sigma (Y_j)
  \lra_{U_\sigma (p_j)} U_\sigma (Y_j^f) \ .\]
  By the proof of \cite[9.7]{DS} we know that~$k^\flat$ is isomorphic
  to~$\gamma_i (\ell^\flat)$ where $ \gamma_i \colon \C_i \ra \Ho \C_i
  $ denotes the localisation functor.  Similarly, $k^\sharp$~is
  isomorphic to~$\gamma_j (\ell^\sharp)$, where~$\gamma_j$ denotes the
  localisation functor for~$ \C_j$, and~$\ell^\sharp$ is adjoint
  to~$\ell^\flat$. In particular, $k^\sharp$~is an isomorphism if and
  only if~$\ell^\sharp$ is a \weq{}. But~$\ell^\sharp$ factors as $
  F_\sigma (Y_i^c) \ra Y_j \rCof^\sim_{p_j} Y_j^f $ which shows
  that~$\ell^\sharp$ is a \weq{} if and only if the homotopy sheaf
  condition (``at~$\sigma$'') holds for~$Y$.

  The second assertion follows immediately since $\hbar$~maps \weqs{}
  to isomorphisms and the property of being a left strict sheaf is
  clearly invariant under isomorphism.
\end{proof}

The theorem applies, for example, to the category of non-linear
sheaves on projective $n$-space. Recall the adjunction
bundle~$\mathfrak{P}^n$ from~\ref{ProjSpaces}. This is an adjunction
bundle of model categories. The resulting category of homotopy sheaves
$\hshv \bigl(\langle n\rangle, \mathfrak{P}^n \bigr)$ is the
category~$\mathbf{P}^n$ of sheaves as defined in~\cite[6.3]{Hu1}.  The
index category~$\langle n\rangle$ is direct with degree function $d
(A) := \# A$.  Hence the $c$-structure exists.  Moreover, all objects
of $\Tw \bigl(\langle n\rangle, \mathfrak{P}^n\bigr)$ are $c$-fibrant.
Thus we can use the construction of Case~1 in
\S\ref{subsec:construction_of_h}, and Theorem~\ref{HomInv}
applies.---More generally, we can consider the category of non-linear
sheaves over any projective toric variety \cite{Hu2}, and
Theorem~\ref{HomInv} identifies non-linear sheaves (or rather,
homotopy classes of such) with ``sheaves in the homotopy category''.

\subsection*{Acknowledgements}

The authors have to thank \textsc{M.~Brun\/} for helpful comments and
suggestions. All diagrams were typeset with \textsc{P.~Taylor\/}'s macro
package \cite{T}.

\raggedright

\bibliographystyle{alpha}
\bibliography{twisted}

\end{document}